\documentclass[12pt]{amsart}
\usepackage{txfonts}
\usepackage{amssymb}
\usepackage{eucal}
\usepackage{graphicx}
\usepackage{amsmath}
\usepackage{amscd}
\usepackage[all]{xy}           

\usepackage{amsfonts,latexsym}
\usepackage{xspace}
\usepackage{epsfig}
\usepackage{float}
\usepackage{color}
\usepackage{fancybox}
\usepackage{colordvi}
\usepackage{multicol}
\usepackage{colordvi}
\usepackage{stmaryrd}
\usepackage[colorlinks,final,backref=page,hyperindex,hypertex]{hyperref}

\newtheorem{theorem}{Theorem}[section]
\newtheorem{defn}[theorem]{Definition}
\newtheorem{lemma}[theorem]{Lemma}
\newtheorem{coro}[theorem]{Corollary}
\newtheorem{prop-def}{Proposition-Definition}[section]

\newtheorem{exam}[theorem]{Example}

\newcommand{\nc}{\newcommand}
\newcommand{\delete}[1]{}

\nc{\mlabel}[1]{\label{#1}}  
\nc{\mcite}[1]{\cite{#1}}  
\nc{\mref}[1]{\ref{#1}}  
\nc{\mbibitem}[1]{\bibitem{#1}} 

\delete{
\nc{\mlabel}[1]{\label{#1}  
{\hfill \hspace{1cm}{\bf{{\ }\hfill(#1)}}}}
\nc{\mcite}[1]{\cite{#1}{{\bf{{\ }(#1)}}}}  
\nc{\mref}[1]{\ref{#1}{{\bf{{\ }(#1)}}}}  
\nc{\mbibitem}[1]{\bibitem[\bf #1]{#1}} 
}

\nc{\bfk}{\mathbf{k}}
\nc{\Der}{\mathrm{Der}}
\nc{\Ker}{\mathrm{Ker}}

\begin{document}

\title{n-Lie bialgebras }

\author{RuiPu  Bai}
\address{College of Mathematics and Information  Science,
Hebei University, Baoding 071002, China} \email{bairuipu@hbu.edu.cn}

\author{Weiwei Guo}
\address{College of Mathematics and Information  Science,
Hebei University, Baoding 071002, China} \email{hubguoweiwei@163.com}

\author{Lixin Lin}
\address{College of Mathematics and Information  Science,
Hebei University, Baoding 071002, China} \email{llxhebeidaxue@163.com}

\author{Yan Zhang}
\address{College of Mathematics and Information  Science,
Hebei University, Baoding 071002, China} \email{zhycn0913@163.com}

\date{\today}

\begin{abstract} The  $n$-Lie bialgebras are studied. In Section 2, the $n$-Lie coalgebra with  rank $r$ is defined, and the structure of it is discussed. In Section 3,  the $n$-Lie bialgebra is introduced. A triple $(L, \mu, \Delta)$ is an $n$-Lie bialgebra if and only if $\Delta$ is a conformal $1$-cocycle on the $n$-Lie algebra $L$ associated to $L$-modules $(L^{\otimes n}, \rho_s^{\mu})$,  $1\leq s\leq n$, and  the structure of $n$-Lie bialgebras is investigated by the structural constants. In Section 4, two-dimensional extension of finite dimensional $n$-Lie bialgebras are studied. For an $m$ dimensional
$n$-Lie bialgebra $(L, \mu, \Delta)$,  and an $ad_{\mu}$-invariant symmetric bilinear form on $L$, the $m+2$ dimensional   $(n+1)$-Lie bialgebra is constructed. In the last section,
the bialgebra structure on the finite dimensional simple $n$-Lie algebra $A_n$ is discussed. It is proved that only    bialgebra structures on the  simple $n$-Lie algebra $A_n$ are rank zero, and rank two.

\end{abstract}

\subjclass[2010]{17B05, 17D99.}

\keywords{ $n$-Lie algebra, $n$-Lie bialgebra,   $n$-Lie coalgebra, two-dimensional extension}

\maketitle



\allowdisplaybreaks

\section{Introduction}

For a given algebraic structure determined by a set of multiplications and a set of relations among the operations, a bialgebra structure on this algebra is obtained
by a corresponding set of comultiplications together with a set of compatibility conditions between the multiplications and comultiplications \cite{Gbook, Vbook}.
A Lie bialgebra \cite{G} consists of a Lie algebra $(L, [ , ])$ and a Lie
 coalgebra $ (L, \Delta)$,  where $ [ , ]: L \wedge L \rightarrow L$ is a Lie bracket,  $\Delta: L\rightarrow L \wedge L$  is a Lie comultiplication, and  the compatibility
condition between the Lie bracket $[  , ,]$ and the Lie comultiplication $\Delta$ is
$$
\Delta([x, y])=(ad x\otimes 1+1\otimes ad x)\Delta(y)-(ad y\otimes 1+1\otimes ad y)\Delta(x), ~~ \forall x, y\in L.
$$

We know that Lie bialgebra is very important, one reason for it  is the Lie bialgebra has a coboundary theory, which
leads to that  the construction of Lie bialgebras has close relation with the solutions of the classical Yang-Baxter equation.

In paper \cite{BGSHh}, authors  studied  local cocycle 3-Lie bialgebras $(A, \mu, \Delta)$, and   defined 3-Lie classical Yang-Baxter equation (3-Lie CYBE)
$$
[ [r, r, r] ]=0, \quad r=\sum_i x_i\otimes y_i\in A\otimes A. $$
From a solution  $r\in A \otimes  A$,  a class of local cocycle 3-Lie bialgebras can be constructed.

Let  $r=\sum\limits_i x_i\otimes y_i\in A\otimes A$ be  a solution of  3-Lie classical Yang-Baxter equation associated to $3$-Lie algebra $(A, \mu)$, then $(L, \mu, \Delta)$ is a local cocycle 3-Lie bialgebra,
where $\Delta: A \rightarrow A\otimes A \otimes A$ defined by for all $x\in A$,
$$
\Delta(x)=\Delta_1(x)+\Delta_2(x)+\Delta_3(x),
$$
where
$\Delta_1(x)=\sum\limits_{ij}([x, x_i, x_j]\otimes y_j\otimes y_i,$  $\Delta_2(x)=\sum\limits_{ij}y_i\otimes [x, x_i, x_j]\otimes y_j,$ $\Delta_3(x)=\sum\limits_{ij}y_j\otimes y_i \otimes [x, x_i, x_j].$

Also Manin triples and matched pairs of $3$-Lie algebras and double construction $3$-Lie bialgebras are introduced in  \cite{BGSHh}. It is proved that the double construction
$3$-Lie bialgebras can be regarded as a special class of  local cocycle $3$-Lie bialgebras.

In paper \cite{BCL}, the $3$-Lie bialgebra $(L, \mu, \Delta)$ is discussed,  the structure of it is different to the local cocycle 3-Lie bialgebras introduced in paper \cite{BGSHh}, it is a special class of local cocycle 3-Lie bialgebras in the case $\Delta_1=\Delta_2=0$.
In \cite{BCL}, structures of $3$-Lie bialgebras are described by the structural constants, and  the complete  classification of $3$-dimensional $3$-Lie bialgebras are provided.

In this paper, we study the finite dimensional  $n$-Lie bialgebras over a field $\mathbb F$ of characteristic zero. In section 2, we discuss the  finite dimensional  $n$-Lie coalgebras.
In section 3, we study the structure of finite dimensional  $n$-Lie bialgebras. In section 4, we investigate  the two dimensional extension  of finite dimensional  $n$-Lie bialgebras. In section 5, we study $n$-Lie bialgebra structure on the simple $n$-Lie algebra $A_n$ over the field of complex numbers.

In the paper, we  suppose that $\mathbb F$ is a field of characteristic zero, $\mathbb Z$ is the set of integers, and $\mathbb Z^+\subset \mathbb Z$ is the set of all positive integers.

For $ i_1, \cdots, i_{n}, j_1, \cdots, j_n\in \mathbb Z^+$, the determinant
$\begin{vmatrix}
\delta_{i_{1}j_{1}} &\cdots & \delta_{i_{n}j_{1}}\\
\vdots &      & \vdots\\
\delta_{i_{1}j_{n}} &\cdots & \delta_{i_{n}j_{n}}\\
\end{vmatrix}$
is simply denoted by $\frac{i_1\cdots i_n}{j_1\cdots j_n}$, where $\delta_{ij}$ is the Kronecher symbol.

Let $V$ be a vector space, $V^*$ be the dual space of $V$. For all $x_1, \cdots, x_n\in V^*$ and $v_1, \cdots, v_n\in V$,
$$\langle x_1\otimes \cdots \otimes x_n, v_1\otimes \cdots \otimes v_n\rangle=\langle x_1, v_1\rangle \cdots \langle x_n, v_n\rangle, $$
where $\langle x_i, v_i\rangle=x_i(v_i), \quad 1\leq i\leq n.$

For a vector space $V$ and $s, n\in \mathbb Z^+$, the $(2n-1)$-ary linear map $\omega_s: V^{\otimes (n-1)}\otimes V^{\otimes n}\rightarrow V^{\otimes (n-1)}\otimes V^{\otimes n}$ is defined as follows, for all $v_1, \cdots, v_{n-1}, u_1, \cdots, u_n\in V$, and $x_1,$ $ \cdots,$ $ x_{n-1}, $ $y_1, $ $\cdots, $ $y_n\in V^*$,

\begin{equation}\label{eq:dualomega}
\langle~~  x_1\otimes \cdots\otimes x_{n-1}\otimes y_1\otimes \cdots\otimes y_n, \omega_s(v_1\otimes \cdots\otimes v_{n-1}\otimes u_1\otimes \cdots\otimes u_n)~~
 \rangle
 \end{equation}
$$=\langle ~~ \omega_s^*(x_1\otimes \cdots\otimes x_{n-1}\otimes y_1\otimes \cdots\otimes y_n), \quad v_1\otimes \cdots\otimes v_{n-1}\otimes u_1\otimes \cdots\otimes u_n~~ \rangle$$
$$\hspace{14mm}=\langle y_1\otimes \cdots\otimes \widehat{y_s}\otimes \cdots \otimes y_n\otimes x_1\otimes \cdots\otimes x_{n-1}\otimes y_s, v_1\otimes \cdots\otimes v_{n-1}\otimes u_1\otimes \cdots\otimes u_n~~ \rangle,$$
where $V^*$ be the dual space of $V$, and $\omega_s^*: V^{*\otimes n-1}\otimes V^{*\otimes n}\rightarrow $ $V^{*\otimes n-1}\otimes V^{*\otimes n}$ is the dual mapping of $\omega_s.$

\section{n-Lie coalgebra}
\mlabel{sec:con}

{\bf An $n$-Lie algebra} $(L,\mu)$ \cite{F} is a vector space $L$ over  $\mathbb F$ endowed with an $n$-multilinear  multiplication
$\mu: A^{\otimes n}\rightarrow A$ satisfying that for all  $x_{i_{1}},\cdots,x_{i_{n-1}},x_{j_{1}},\cdots,x_{j_{n}}\in L,$

\begin{equation}\label{eq:skew1}
\mu(x_1, \cdots, x_n)=sign(\sigma)~\mu(x_{\sigma(1)}, \cdots, x_{\sigma(n)}),
\end{equation}

\begin{equation}\label{eq:jacobi}
\mu(x_{i_{1}},\cdots,x_{i_{n-1}},\mu(x_{j_{1}},\cdots,x_{j_{n}}))=\sum_{s=1}^{n}(-1)^{n-s}\mu(x_{j_{1}},\cdots,\widehat{x_{j_{s}}},\cdots,x_{j_{n}},
\mu(x_{i_{1}},\cdots,x_{i_{n-1}},x_{j_{s}})),
\end{equation}
where $\sigma\in S_n$ and the symbol $\widehat{x_{j_s}}$ means that $x_{j_s}$ is omitted. In the case $n\geq 3, $ the identity Eq.\eqref{eq:jacobi} is usually called $n$-Jacobi identity, or Filipov identity.

The Eq.\eqref{eq:skew1} and Eq.\eqref{eq:jacobi} can be respectively described as follows

\begin{equation}\label{eq:skew}\mu(1-\tau)=0,
\end{equation}

\begin{equation}\label{eq:jacobi2}
\mu(\underbrace{1\otimes1\otimes\cdots\otimes1}_{n-1}\otimes\mu)(1-(-1)^{n-s}\sum_{s=1}^{n}\omega_{s})=0,
\end{equation}
where $1$ is identity,  $\omega_s$ is defined as Eq.\eqref{eq:dualomega}, for $1\leq s\leq n$, and  $\tau:L^{\otimes n}\rightarrow L^{\otimes n}$ is defined as  for all  $x_{i_{1}},$ $\cdots,$ $x_{i_n}\in L,$
$$\tau(x_{i_{1}}\otimes x_{i_{2}}\otimes\cdots\otimes x_{i_{n}})=sign(\sigma)x_{i_{\sigma (1)}}\otimes x_{i_{\sigma(2)}}\otimes\cdots\otimes x_{i_{\sigma(n)}}, \forall  \sigma\in S_{n}.$$

We give the definition of $n$-Lie coalgebras.

\begin{defn}\label{def:coalg}  An $n$-Lie coalgebra $(L, \Delta)$ is a vector space $L$ with a linear mapping $\Delta: L\rightarrow L^{\wedge n}$ satisfying that

\begin{equation}\label{eq:coeq}
(1-\sum\limits_{s=1}^{n}(-1)^{n-s}\omega_{s})(\underbrace{1\otimes1\otimes\cdots\otimes1}_{n-1}\otimes\Delta)\Delta=0,
\end{equation}
 where $1$ is identity, and $\omega_s$ is defined as Eq.\eqref{eq:dualomega} for $1\leq s\leq n$.
 \end{defn}

In Eq.\eqref{eq:coeq}, the linear mapping $1\otimes \cdots \otimes 1\otimes \Delta:$ $ L^{\otimes n}\rightarrow L^{\otimes 2n-1}$ satisfies that,  for all $x_k\in L$ and
 $x^{l_1}, $ $\cdots, $ $x^{l_{n-1}}, $ $x^{s_1}, $ $\cdots, $ $x^{s_n}\in L^*,$
\begin{equation}\label{eq:defco}
\langle x^{l_1}\otimes \cdots\otimes x^{l_{n-1}}\otimes x^{s_1}\otimes\cdots\otimes x^{s_n}, (1\otimes \cdots \otimes 1\otimes \Delta)(x_1\otimes \cdots \otimes x_n)\rangle
\end{equation}
$$\hspace{-1.5cm}=\langle x^{l_1}\otimes \cdots\otimes x^{l_{n-1}} \otimes x^{s_1}\otimes \cdots \otimes x^{s_n},  x_1\otimes \cdots \otimes x_{n-1}\otimes \Delta(x_n)\rangle$$
$$\hspace{-2.8cm}=\langle x^{l_1}\otimes \cdots\otimes x^{l_{n-1}}\otimes \Delta^*(x^{s_1},\cdots, x^{s_n}), x_1\otimes \cdots \otimes x_n\rangle.$$

For describing the structure of $n$-Lie coalgebras, we need structural  constants of $n$-Lie algebras.

Let $(L,\mu)$ be an $m$-dimensional $n$-Lie algebra with a basis $x_{1},x_{2},\cdots,x_{m}.$ Suppose

\begin{equation}\label{eq:Lcc1}
\mu(x_{i_1}, \cdots, x_{i_n})=\sum\limits_{k=1}^m c^k_{i_1\cdots i_n}x_k, \quad 1\leq k\leq m, ~~  1\leq i_1, \cdots, i_n\leq m.
\end{equation}

By Eq.\eqref{eq:skew}, $c_{i_{1}\cdots i_{n}}^{k}$ satisfies that
 \begin{equation}\label{eq:Lset1}
 c_{i_{1}\cdots i_{n}}^{k}=sign(\sigma)c_{i_{\sigma(1)}i_{\sigma(2)}\cdots i_{\sigma(n)}}^{k}, \quad \forall \sigma\in S_n, ~~ 1\leq i_{1}, i_{2}, \cdots, i_{n}\leq m.
 \end{equation}

For all $1\leq i_1, \cdots, i_{n-1}\leq m$ and $1\leq j_1,  \cdots,   j_n\leq m$,

$
\mu(x_{i_{1}},\cdots,x_{i_{n-1}},\mu(x_{j_{1}},\cdots,x_{j_{n}}))=\mu(x_{i_{1}},\cdots,x_{i_{n-1}},\sum\limits_{t=1}^{m}c_{j_{1}\cdots j_{n}}^{t}x_{t})
$\\
$=\sum\limits_{k=1}^{m}\sum\limits_{t=1}^{m}c_{j_{1}\cdots j_{n}}^{t}c_{i_{1}\cdots i_{n-1}t}^{k}x_{k},
$

\vspace{2mm}$
\sum\limits_{s=1}^{n}(-1)^{n-s}\mu(x_{j_{1}},\cdots,\widehat{x_{j_{s}}},\cdots,x_{j_{n}},\mu(x_{i_{1}},\cdots,x_{i_{n-1}},x_{j_{s}}))$
\\$=\sum\limits_{s=1}^{n}(-1)^{n-s}\mu(x_{j_{1}},\cdots,\widehat{x_{j_{s}}},\cdots,x_{j_{n}},\sum\limits_{t=1}^{m}c_{i_{1}\cdots i_{n-1}j_{s}}^{t}x_{t})$
\\
$=\sum\limits_{t=1}^{m}\sum\limits_{s=1}^{n}(-1)^{n-s}c_{i_{1}\cdots i_{n-1}j_{s}}^{t}\mu(x_{j_{1}},\cdots,\widehat{x_{j_{s}}},\cdots,x_{j_{n}},x_{t})
$
\\
$=\sum\limits_{k=1}^{m}\sum\limits_{t=1}^{m}\sum\limits_{s=1}^{n}(-1)^{n-s}c_{i_{1}\cdots i_{n-1}j_{s}}^{t}c^k_{j_{1}\cdots j_{ s-1}j_{s+1}\cdots j_{n}t}x_k,
$

\vspace{2mm}\noindent  by  Eq.\eqref{eq:jacobi}, we obtain that

\begin{equation}\label{eq:lset2}
\sum\limits_{t=1}^mc_{j_{1}\cdots j_{n}}^{t}c_{i_{1}\cdots i_{n-1}t}^{k}-\sum\limits_{t=1}^m\sum\limits_{s=1}^{n}(-1)^{n-s}c_{i_{1}\cdots i_{n-1}j_{s}}^tc^k_{j_{1}\cdots j_{ s-1}j_{s+1}\cdots j_{n}t}=0.
\end{equation}

This shows the following result.

\begin{theorem}
Let $L$ be a vector space with a basis $x_{1},x_{2},\cdots,x_{m}$, the $n$-ary linear multiplication $\mu: L^{\otimes n} \rightarrow L$ be  defined as Eq.\eqref{eq:Lcc1}.
Then  $(L, \mu)$ is an $n$-Lie algebra if and only $\mu$ satisfies Eq.\eqref{eq:Lset1} and Eq.\eqref{eq:lset2}.
\label{thm:Lcons}
\end{theorem}

\begin{theorem}\label{thm:LandC}
Let $L$ be a vector space over a field $\mathbb F$, $\Delta: L\rightarrow L^{\otimes n}$ be a linear map. Then $(L, \Delta)$ is an $n$-Lie coalgebra if and only if ~ $(L^*, \Delta^*)$ ~ is an $n$-Lie algebra.
\end{theorem}

\begin{proof} Let $L^*$  be the dual space of $L$, $x^1, \cdots, x^m$ be the dual basis of $x_1, \cdots, x_m$,
that is, $\langle x_i, x^j\rangle=\delta_{ij}$, $1\leq i, j\leq m. $ Let $\Delta^*: L^{*\otimes n} \rightarrow L^*$ be the dual mapping of $\Delta$, that is, for all $y^1, \cdots, y^n\in L^*$,
and $x\in L$,
\begin{equation}\label{eq:dual}
\langle y^1 \otimes \cdots \otimes y^n, ~~ \Delta(x)\rangle=\langle \Delta^*(y^1,  \cdots, y^n), x\rangle.
\end{equation}

Suppose
$$
\Delta^*(x^{i_1}, \cdots, x^{i_n})=\sum\limits_{k=1}^m c^k_{i_1\cdots i_n}x^k, \quad 1\leq k\leq m, ~~  1\leq i_1, \cdots, i_n\leq m.
$$
and
\begin{equation}\label{eq:Lcc2}
\Delta(x_k)=\sum\limits_{i_1\cdots i_n}a^{i_1\cdots i_n}_kx_{i_1}\otimes \cdots \otimes x_{i_n}, \quad 1\leq i_1, \cdots, i_n\leq m, 1\leq k\leq m.
\end{equation}
Then for all $\quad 1\leq i_1, \cdots, i_n\leq m, \quad  1\leq j\leq m,$

$$\langle \Delta^*(x^{i_1},  \cdots,  x^{i_n}), x_j\rangle=\langle \sum\limits_{k=1}^mc_{i_1\cdots i_n}^k x^k, x_j\rangle=c^j_{i_1\cdots i_n}.$$

Since

$\langle \Delta^*(x^{i_1},  \cdots,  x^{i_n}), x_j\rangle=\langle x^{i_1}\otimes \cdots\otimes  x^{i_n}, \Delta(x_j)\rangle
$
\\
$=\langle x^{i_1}\otimes \cdots\otimes  x^{i_n}, \sum\limits_{j_1\cdots j_n} a^{j_1\cdots j_n}_j(x_{j_1} \otimes \cdots \otimes x_{j_n}\rangle
$
$=a^{i_1\cdots i_n}_j$,
 \\we obtain that

\begin{equation}\label{eq:CCL}
a_k^{i_1\cdots i_n}=c^k_{i_1\cdots i_n}, \quad 1\leq i_1, \cdots, i_n\leq m, \quad 1\leq k\leq m.
\end{equation}

Therefore, the result holds.

\end{proof}

In fact, if $(L, \Delta)$ is an $n$-Lie coalgebra, by Eq.\eqref{eq:CCL} and Eq.\eqref{eq:coeq}, for all $1\leq l_1< \cdots l_{n-1}\leq m$ and $1\leq s_1 < \cdots < s_n\leq m,$

 \vspace{2mm}$\langle ~ x^{l_1}\otimes \cdots\otimes x^{l_{n-1}}\otimes x^{s_1}\otimes\cdots\otimes x^{s_n}, (1\otimes \cdots \otimes 1\otimes \Delta)\Delta(x_k) ~ \rangle$
 \\
 $=\langle ~ x^{l_1}\otimes \cdots\otimes x^{l_{n-1}}\otimes x^{s_1}\otimes\cdots\otimes x^{s_n}, (1\otimes \cdots \otimes 1\otimes \Delta)\Big(\sum\limits_{i_1\cdots i_n}a^{i_1\cdots i_n}_kx_{i_1}\wedge \cdots\wedge x_{i_n}\Big) ~ \rangle$
 \\
  $=\langle x^{l_1}\otimes \cdots\otimes x^{l_{n-1}}\otimes \Delta^*(x^{s_1}\otimes\cdots\otimes x^{s_n}),\sum\limits_{i_1\cdots i_n}a^{i_1\cdots i_n}_kx_{i_1}\wedge \cdots\wedge x_{i_n}\rangle$
 \\
$=\sum\limits_{i_1\cdots i_n}a^{i_1\cdots i_n}_k \sum\limits_{t=1}^m c_{s_1\cdots s_n}^t\langle x^{l_1}\otimes \cdots\otimes x^{l_{n-1}}\otimes x^t, x_{i_1}\wedge \cdots\wedge x_{i_n}\rangle$
 \\
$=\sum\limits_{t=1}^mc_{s_1\cdots s_n}^t\sum\limits_{i_1\cdots i_n}a^{i_1\cdots i_n}_k \frac{l_1\cdots l_{n-1}t}{i_1 \cdots i_n}$
 $=\sum\limits_{t=1}^m c_{s_1\cdots s_n}^ta^{l_1\cdots l_{n-1}t}_k=\sum\limits_{t=1}^m c_{s_1\cdots s_n}^tc_{l_1\cdots l_{n-1}t}^k$

\vspace{5mm}\noindent$=\langle ~x^{l_1}\otimes \cdots\otimes x^{l_{n-1}}\otimes x^{s_1}\otimes\cdots\otimes x^{s_n}, \sum\limits_{r=1}^{n}(-1)^{n-r}\omega_{r}(1\otimes1\otimes\cdots\otimes1\otimes\Delta)\Delta(x_k)~ \rangle$
\\
$=\langle ~\sum\limits_{r=1}^{n}(-1)^{n-r}\omega_{r}^*(x^{l_1}\otimes \cdots\otimes x^{l_{n-1}}\otimes x^{s_1}\otimes\cdots\otimes x^{s_n}), (1\otimes1\otimes\cdots\otimes1
\otimes\Delta)\Delta(x_k) ~\rangle$
\\
$
=\langle ~\sum\limits_{r=1}^{n}(-1)^{n-r}(x^{s_1}\otimes \cdots \otimes \widehat{x^{s_r}}\otimes \cdots x^{s_{n}}\otimes \Delta^*(x^{l_1},\cdots, x^{l_{n-1}}, x^{s_r}), \sum\limits_{i_1\cdots i_n} a^{i_1\cdots i_n}_kx_{i_1}\wedge\cdots\wedge x_{i_n} ~ \rangle
$
\\
$
= \sum\limits_{r=1}^{n}(-1)^{n-r}\sum\limits_{i_1\cdots i_n} a^{i_1\cdots i_n}_k\sum\limits_{t=1}^m c^t_{l_1\cdots l_{n-1}s_r}\langle x^{s_1}\otimes \cdots \otimes \widehat{x^{s_r}}\otimes \cdots x^{s_{n}}\otimes x^t, x_{i_1}\wedge\cdots\wedge x_{i_n}\rangle
$
\\
$
= \sum\limits_{t=1}^m\sum\limits_{r=1}^{n}(-1)^{n-r}\sum\limits_{i_1\cdots i_n} a^{i_1\cdots i_n}_k c^t_{l_1\cdots l_{n-1}s_r} \frac{s_1 \cdots  s_{r-1} s_{r+1}\cdots s_{n} t}{ i_1\cdots i_n}
$
\\
$
= \sum\limits_{t=1}^m\sum\limits_{r=1}^{n}(-1)^{n-r} a^{s_1 \cdots  s_{r-1} s_{r+1}\cdots s_{n} t}_k c^t_{l_1\cdots l_{n-1}s_r}
$
\\
$
= \sum\limits_{t=1}^m\sum\limits_{r=1}^{n}(-1)^{n-r} c_{s_1 \cdots  s_{r-1} s_{r+1}\cdots s_{n} t}^k c^t_{l_1\cdots l_{n-1}s_r}.
$

Therefore, $(L^*, \Delta^*)$ is an $n$-Lie algebra.

Let $(L, \Delta)$ be an $n$-Lie coalgebra.  If the dimension of the derived algebra $L^{*1}$  of $n$-Lie algebra $(L^*, \Delta^*)$ is $r$, then $n$-Lie coalgebra $(L, \Delta)$
 is referred to as {\bf an rank $r$,  and is denoted by $R(\Delta)=r.$}

\begin{defn}\label{def:coiso}   Let $(L_{1},\Delta_{1})$ and $(L_{2},\Delta_{2})$ be $n$-Lie coalgebras. If there is a linear isomorphism
$\varphi:L_{1}\rightarrow L_{2}$ satisfying that, for all $x\in L_{1},$
\begin{equation}\label{eq:coiso}
\underbrace{(\varphi\otimes\cdots\otimes\varphi}_{n})(\Delta_{1}(x))=\Delta_{2}(\varphi(x)),
\end{equation}
then $(L_{1},\Delta_{1})$ is isomorphic to $(L_{2},\Delta_{2}),$ and $\varphi$ is called an $n$-Lie coalgebra isomorphism,where
$$(\underbrace{\varphi\otimes\cdots\otimes\varphi}_{n})\sum_{i}(x_{1_{i}}\otimes x_{2_{i}}\otimes\cdots\otimes x_{n_{i}})=\sum_{i}\varphi(x_{1_{i}})\otimes\varphi(x_{2_{i}})\otimes\cdots\otimes\varphi(x_{n_{i}}).$$
\end{defn}

\begin{theorem}\label{thm:LCL}  Let $(L_{1},\Delta_{1})$ and $(L_{2},\Delta_{2})$ be $n$-Lie coalgebras. Then $\varphi:L_{1}\rightarrow L_{2}$ is an
$n$-Lie coalgebra isomorphism from $(L_{1},\Delta_{1})$ to $(L_{2},\Delta_{2})$ if and only if the dual mapping $\varphi^{\ast}:L_{2}^{\ast}\rightarrow L_{1}^{\ast}$ is an $n$-Lie algebra isomorphism from $(L_{2}^{\ast},\Delta_{2}^{\ast})$ to $(L_{1}^{\ast},\Delta_{1}^{\ast}),$ where for $\xi\in L_{2}^{\ast},$ $v\in L_{1},$
$\langle\varphi^{\ast}(\xi),v\rangle=\langle\xi,\varphi(v)\rangle.$
\end{theorem}

\begin{proof} Since $(L_{1},\Delta_{1})$ and $(L_{2},\Delta_{2})$ are $n$-Lie coalgebras, thanks to Theorem \ref{thm:LandC} $(L_{1}^{\ast},\Delta_{1}^{\ast})$ and $(L_{2}^{\ast},\Delta_{2}^{\ast})$ are $n$-Lie algebras. If  $\varphi:L_{1}\rightarrow L_{2}$  is an
$n$-Lie coalgebra isomorphism,  then the dual mapping
$\varphi^{\ast}:L_{2}^{\ast}\rightarrow L_{1}^{\ast}$ is a linear isomorphism, and for all
$f_{k_{1}},\cdots,f_{k_{n}}\in L_{2}^{\ast},$ $x\in L_{1}^{\ast},$ by Eqs. \eqref{eq:skew} and \eqref{eq:dual}

$
\langle\varphi^{\ast}\Delta_{2}^{\ast}(f_{k_{1}},\cdots,f_{k_{n}}),x\rangle=\langle\Delta_{2}^{\ast}(f_{k_{1}},\cdots,f_{k_{n}}),\varphi(x)\rangle$
$
=\langle f_{k_{1}}\otimes\cdots\otimes f_{k_{n}},\Delta_{2}(\varphi(x))\rangle
$
\\
$
=\langle f_{k_{1}}\otimes\cdots\otimes f_{k_{n}},(\underbrace{\varphi\otimes\cdots\otimes\varphi}_{n})(\Delta_{1}(x))\rangle$
$=\langle \varphi^{\ast}(f_{k_{1}})\otimes\cdots\otimes \varphi^{\ast}(f_{k_{n}}),\Delta_{1}(x)\rangle$
\\
$=\langle\Delta_{1}^{\ast}(\varphi^{\ast}(f_{k_{1}}),\varphi^{\ast}(f_{k_{2}}),\cdots,\varphi^{\ast}(f_{k_{n}})),x\rangle.$

Therefore, $\varphi^{\ast}\Delta_{2}^{\ast}(f_{k_{1}},\cdots,f_{k_{n}})=\Delta_{1}^{\ast}(\varphi^{\ast}(f_{k_{1}}),
\varphi^{\ast}(f_{k_{2}}),\cdots,\varphi^{\ast}(f_{k_{n}})),$
that is, $\varphi^{\ast}:L_{2}^{\ast}\rightarrow L_{1}^{\ast}$ is an $n$-Lie algebra isomorphism from $(L_{2}^{\ast},\Delta_{2}^{\ast})$ to $(L_{1}^{\ast},\Delta_{1}^{\ast}).$

Similar discussion, we obtain the conversion. The proof is complete.
\end{proof}

\begin{exam}\label{exam:comatrix}
Now we give an example of $3$-Lie coalgebra. Let $V=M(m, \mathbb F)$ be a vector space of $(m\times m)$-order matrices over a field $\mathbb F$. Then $V$ has a basis
$$\Big\{ E_{ij}~ | ~ 1\leq i\neq j\leq m\Big\}\cup\Big\{ E_{jj}-E_{j+1,j+1}, ~ E=\sum\limits_{i+1}^m E_{ii}~ |~1\leq j\leq m-1\Big\},$$
where $E_{ij}$ is matrix unit.
Define linear mapping $\Delta: V \rightarrow V^{\wedge  3}$, by
$$
\Delta(E_{ij})=\sum\limits_{k=1}^mE_{kk}\wedge E_{ik}\wedge E_{kj}, ~ 1\leq i\neq j\leq m, \quad \Delta(E)=0,$$

$$\Delta(E_{ii}-E_{i+1,i+1})=\sum\limits_{k=1}^mE_{kk}\wedge E_{i,i+1}\wedge E_{i+1, i}, ~  1\leq i\leq m-1.
$$
Then by a direct computation,   $(V, \Delta)$ is a $3$-Lie coalgebra.

\end{exam}

\begin{exam}\label{exam:con+1n}
Let  $V$ be an $(n+1)$-dimensional vector space  with a basis $e_1, \cdots, e_{n+1}.$
Define linear map $\Delta: V \rightarrow V^{\otimes n}$, by
$$
\Delta(e_i)=e_1\wedge \cdots\wedge \widehat{e_i}\wedge \cdots\wedge e_{n+1}, ~ 1\leq i\leq n+1.
$$
Then $(V, \Delta)$ is an $n$-Lie coalgebra with rank $n+1$.

\end{exam}

\section{n-Lie bialgebra}
\mlabel{sec:bialg}

In this section we discuss $n$-Lie bialgebras.

 Let $(L, \mu)$ be an $n$-Lie algebra,  $(V, \rho_i)$, ~ $i=1, \cdots, l$, ~ be  representations of $(L, \mu)$, and $f: L\rightarrow V$ be a linear mapping.
If $f$ satisfies that for all $x_1, \cdots, x_n\in L,$

\begin{equation}\label{eq:cocyc1}
f(\mu(x_1, \cdots, x_n))=\sum\limits_{k=1}^n(-1)^{n-k}\rho_i(x_1, \cdots, \widehat{x_k}, \cdots, x_n)f(x_k),
\end{equation}
then $f$ is called { \bf an $1$-cocycle on $L$ associated  to $(V, \rho_i)$.}

If $f$ satisfies that for all $x_1, \cdots, x_n\in L,$
\begin{equation}\label{eq:cocyc2}
f(\mu(x_1, \cdots, x_n))=\sum\limits_{i=1}^l\sum\limits_{k=1}^n(-1)^{n-k}\rho_i(x_1, \cdots, \widehat{x_k}, \cdots, x_n)f(x_k),
\end{equation}
then $f$ is called { \bf a conformal $1$-cocycle on $L$ associated to $(V, \rho_i)$, $1\leq i\leq l$.}

Let $(L, \mu)$ be an $n$-Lie algebra, for $1\leq s\leq n$, define linear mapping $\rho^{\mu}_s: L^{\wedge n-1}\rightarrow End(L^{\otimes n}),$
$$
\rho^{\mu}_s=\underbrace{1\otimes \cdots \otimes 1}_{s-1}\otimes ad_{\mu} \otimes \underbrace{1\otimes \cdots\otimes 1}_{n-s},
$$
for all $x_1, \cdots, x_{n-1}, y_1, \cdots, y_n\in L,$
\begin{equation}\label{eq:rho}
\rho^{\mu}_s(x_1,\cdots, x_{n-1})=\underbrace{1\otimes \cdots \otimes 1}_{s-1}\otimes ad_{\mu}(x_1, \cdots, x_{{n-1}})\otimes \underbrace{1\otimes \cdots\otimes 1}_{n-s},
\end{equation}
$$
\rho^{\mu}_s(x_1,\cdots, x_{n-1})(y_1, \cdots, y_n)=(\underbrace{1\otimes \cdots \otimes 1}_{s-1}\otimes ad_{\mu}(x_1, \cdots, x_{{n-1}})\otimes \underbrace{1\otimes \cdots\otimes 1}_{n-s})(y_1, \cdots, y_n)
$$

\hspace{5cm}$
=y_1\otimes \cdots \otimes y_{s-1}\otimes \mu(x_1, \cdots, x_{n-1}, y_s)\otimes y_{s+1}\otimes \cdots \otimes y_n.
$

We have the following result.

\begin{theorem}\label{thm:rhorepre}

Let $(L, \mu)$ be an $n$-Lie algebra, and $\rho^{\mu}_s$, $1\leq s\leq n$ are defined as Eq.\eqref{eq:rho}. Then $(L^{\otimes n}, \rho^{\mu}_s)$ for $1\leq s\leq n$
are modules of  $n$-Lie algebra $(L, \mu)$.

\end{theorem}

\begin{proof} For all $x_1, \cdots, x_{n-1}, z_1, \cdots, z_{n-1}, y_1, \cdots, y_n\in L,$ and $1\leq s\leq n$, by Eq.\eqref{eq:rho},

\vspace{3mm}$[\rho^{\mu}_s(x_1, \cdots, x_{n-1}), \rho^{\mu}_s(z_1, \cdots, z_{n-1})]$
\\
$
=\underbrace{1\otimes \cdots\otimes 1}_{s-1}\otimes
[ad_{\mu}(x_1, \cdots, x_{{n-1}}), ad_{\mu}(z_1, \cdots, z_{{n-1}})]\otimes \underbrace{1\cdots \otimes 1}_{n-s}$
\\
$
=\underbrace{1\otimes \cdots\otimes 1}_{s-1}\otimes
\sum\limits_{t=1}^{n-1}ad_{\mu}(z_1, \cdots, \mu(x_1, \cdots, x_{n-1}, z_t), \cdots, z_{{n-1}}) \otimes \underbrace{1\cdots \otimes 1}_{n-s}$
\\
$
=\sum\limits_{t=1}^{n-1}(\underbrace{1\otimes \cdots\otimes 1}_{s-1}\otimes
ad_{\mu}(z_1, \cdots, \mu(x_1, \cdots, x_{n-1}, z_t), \cdots, z_{{n-1}}) \otimes \underbrace{1\cdots \otimes 1}_{n-s})$
\\
$
=\sum\limits_{t=1}^{n-1}\rho^{\mu}_s(z_1, \cdots, \mu(x_1, \cdots, x_{n-1}, z_t), \cdots, z_{{n-1}}).$

\vspace{5mm}
$\rho^{\mu}_s(\mu(y_1, \cdots, y_n), x_1, \cdots, x_{n-2})$
\\
$
=\underbrace{1\otimes \cdots\otimes 1}_{s-1}\otimes
ad_{\mu}(\mu(y_1, \cdots, y_n), x_1, \cdots, x_{n-2})\otimes \underbrace{1\cdots \otimes 1}_{n-s}$
\\
$
=\underbrace{1\otimes \cdots\otimes 1}_{s-1}\otimes
\sum\limits_{t=1}^n (-1)^{n-t}ad_{\mu}(y_1, \cdots, \widehat{y_t}, \cdots, y_n)ad_{\mu}(x_1, \cdots, x_{n-2},y_t)\otimes \underbrace{1\cdots \otimes 1}_{n-s}$
\\
$
=\sum\limits_{t=1}^n (-1)^{n-t}\rho^{\mu}_s(y_1, \cdots, \widehat{y_t}, \cdots, y_n)\rho^{\mu}_s(x_1, \cdots, x_{n-2},y_t).$
\\
The result follows.

\end{proof}

Now we  give the definition of $n$-Lie bialgebra.

\begin{defn}\label{def:bialg}
An $n$-Lie bialgebra is a triple $(L,\mu,\Delta)$ which satisfies that\\
$(1)$ $(L,\mu)$ is an $n$-Lie algebra,\\
$(2)$ $(L,\Delta)$ is an $n$-Lie coalgebra,\\
$(3)$ $\Delta$ is a conformal qusi-$1$-cocycle on $n$-Lie algebra $(L, \mu)$ associative to $(L^{\otimes n}, \rho^{\mu}_s)$  for $s=1, \cdots, n$.
\end{defn}

From Eq. \eqref{eq:cocyc2}, Definition \ref{def:bialg} and Theorem \ref{thm:rhorepre}, $\mu, \Delta$ satisfy that for all, $x_1, \cdots, x_n\in L,$

\begin{equation}\label{eq:bialg1}
\Delta\mu(x_{1},\cdots,x_{n})=\sum\limits_{s=1}^{n}\sum\limits_{k=1}^n(-1)^{n-k}\rho^{\mu}_s(x_1, \cdots, \widehat{x_k}, \cdots, x_n)\Delta(x_k)
\end{equation}

\hspace{2cm}$
=\sum\limits_{s=1}^{n}\sum\limits_{k=1}^n(-1)^{n-k}(\underbrace{1\otimes \cdots \otimes 1}_{s-1}\otimes ad_{\mu}(x_1, \cdots, \widehat{x_k}, \cdots, x_{{n}})\otimes \underbrace{1\otimes \cdots\otimes 1}_{n-s})\Delta(x_k).
$

\begin{exam} Let $L$ be an $(n+1)$-dimensional vector space with a basis $x_{1},x_{2},\cdots,x_{n+1}.$ Define linear multiplication
$\mu: L^{\wedge n}\rightarrow L,$  and $\Delta: L\rightarrow L^{\wedge {n}},$  by

$$\mu(x_{1},x_{3},\cdots,x_{n+1})=x_{1}, \quad \mu(x_{2},x_{3},\cdots,x_{n+1})=x_{2}, \quad ~~\text{ and others are zero,}$$

$$\Delta(x_{1})=x_{3}\wedge x_{2} \wedge x_{4}\wedge\cdots\wedge x_{n+1}, \quad
\Delta(x_{3})=x_{1}\wedge x_{2} \wedge x_{4}\wedge\cdots\wedge x_{n+1}, ~\Delta(x_{j})=0, j\neq 1, 3.$$
Then the triple $(L,\mu,\Delta)$ is an $(n+1)$-dimensional $n$-Lie bialgebra.
\end{exam}

In fact, by a direct computation, $(L, \mu)$ and $(L^*, \Delta^*)$ are $n$-Lie algebras.
From

\vspace{2mm}$\Delta\mu(x_{1},x_{3},x_{4}\cdots,x_{n+1})=\Delta(x_{1})=x_{3}\wedge x_{2} \wedge x_{4}\wedge\cdots\wedge x_{n+1},$

\vspace{2mm}$\sum\limits_{s=1}^{n}(-1)^{n-1}\rho^{\mu}_s(x_3, \cdots,  x_{n+1})\Delta(x_1)+$$\sum\limits_{s=1}^{n}\sum\limits_{k=2}^{n}(-1)^{n-k}\rho^{\mu}_s(x_1, x_3, \cdots, \widehat{x_{k+1}}, \cdots  x_{n+1})\Delta(x_{k+1})
$
\\
$
=\sum\limits_{s=1}^n(-1)^{n-1}(1\otimes \cdots\otimes ad_{\mu}(x_3, \cdots, x_{n+1})\otimes \cdots\otimes 1)(x_{3}\wedge x_{2} \wedge x_{4}\wedge\cdots\wedge x_{n+1})
$
\\
$
+\sum\limits_{s=1}^n\sum\limits_{k=2}^n(-1)^{n-k}(1\otimes \cdots\otimes ad_{\mu}(x_1, x_3, \cdots, \widehat{x_{k+1}}, \cdots, x_{n+1})\otimes \cdots\otimes 1)\Delta(x_{k+1})$
\\
$
=\sum\limits_{s=1}^n(-1)^{n-1}(1\otimes \cdots\otimes 1\otimes ad_{\mu}(x_3, \cdots, x_{n+1})\otimes 1\otimes \cdots\otimes 1)(x_{3}\wedge x_{2} \wedge x_{4}\wedge\cdots\wedge x_{n+1})
$
\\
$
+\sum\limits_{s=1}^n(-1)^{n}(1\otimes \cdots \otimes 1\otimes ad_{\mu}(x_1, x_4, \cdots, x_{n+1})\otimes \cdots\otimes 1)(x_1\wedge x_2 \wedge x_{4}\wedge\cdots\wedge x_{n+1})$
\\
$
=x_3\wedge x_2 \wedge x_{4}\wedge\cdots\wedge x_{n+1}.
$

\vspace{2mm}$\Delta\mu(x_{2},x_{3},x_{4}\cdots,x_{n+1})=\Delta(x_{2})=0,$

\vspace{2mm}$\sum\limits_{s=1}^{n}\sum\limits_{k=1}^{n}(-1)^{n-k}\rho^{\mu}_s(x_2, \cdots, \widehat{x_{k+1}}, \cdots,  x_{n+1})\Delta(x_{k+1})
$
$
=\sum\limits_{s=1}^{n}(-1)^{n}\rho^{\mu}_s(x_2,  x_4, \cdots,  x_{n+1})\Delta(x_{3})
$
\\
$
=\sum\limits_{s=1}^{n}(-1)^{n}\rho^{\mu}_s(x_2,x_4, \cdots,  x_{n+1})(x_1\wedge x_2\wedge x_4\wedge\cdots\wedge x_{n+1})=0.
$

\vspace{2mm}For subset $\{ y_1, \cdots,, y_n\}$ is equal to subset $\{ x_1, x_2, x_4, \cdots, x_{n+1}\},$

$$\Delta\mu(y_1, \cdots, y_{n})=\sum\limits_{s=1}^{n}\sum\limits_{k=1}^{n}(-1)^{n-k}\rho^{\mu}_s(y_1, \cdots, \widehat{y_{k}}, \cdots,  y_{n})\Delta(y_{k})=0.
$$

Therefore, $(L, \mu, \Delta)$ is an $n$-Lie bialgebra.

Now we describe the structure of $n$-Lie bialgebras by means of structural constants.

\begin{theorem}\label{thm:bialgconst}
Let $ L $ be an $m$-dimensional  vector space over a field $\mathbb F$ with a basis $ x_1, \cdots, x_m $, \quad $\mu: L^{\wedge n}\rightarrow L$ \quad and \quad $\Delta: L \rightarrow L^{\wedge n}$ be linear maps, and suppose
\begin{equation}\label{eq:biconst}
\mu(x_{i_{1}},x_{i_{2}},\cdots,x_{i_{n}})=\sum_{l=1}^{m}c_{i_{1}\cdots i_{n}}^{l}x_{l}, \quad \Delta(x_{l})=\sum_{i_1 \cdots i_n}a_{l}^{i_{1}\cdots i_{n}}x_{i_{1}}\otimes \cdots\otimes  x_{i_{n}},
\end{equation}
where $c_{i_{1}\cdots i_{n}}^{l}, a_{l}^{j_{1}\cdots j_{n}}\in \mathbb F,$ $1\leq i_{1}, \cdots, i_{n}, j_1, \cdots, j_n\leq m,$
and  for all $\sigma\in S_n$,
$$c_{i_{1}\cdots i_{n}}^{l}=sign ~ c_{i_{\sigma(1)}\cdots i_{\sigma(n)}}^{l},\quad  a_{l}^{i_{1}\cdots i_{n}}=sign ~ a_{l}^{i_{\sigma(1)}\cdots i_{\sigma(n)}}.$$
Then $(L, \mu, \Delta)$ is an $n$-Lie bialgebra if and only if  constants
$$\Big\{~~ c_{i_1\cdots i_n}^l, \quad a_{l}^{i_{1}\cdots i_{n}}~| ~ 1\leq i_1, \cdots i_n, l\leq m ~~ \Big\} $$ satisfy that  for all $1\leq i_1< \cdots < i_n\leq m, $ and $1\leq j_1 < \cdots < j_n\leq m$,

\begin{equation}\label{eq:lset3}
\sum\limits_{t=1}^mc_{j_{1}\cdots j_{n}}^{t}c_{i_{1}\cdots i_{n-1}t}^{k}-\sum\limits_{t=1}^m\sum\limits_{s=1}^{n}(-1)^{n-s}c_{i_{1}\cdots i_{n-1}j_{s}}^tc^k_{j_{1}\cdots j_{ s-1}j_{s+1}\cdots j_{n}t}=0,
\end{equation}
\begin{equation}\label{eq:lset4}
\sum\limits_{t=1}^m a^{j_{1}\cdots j_{n}}_{t}a^{i_{1}\cdots i_{n-1}t}_{k}-\sum\limits_{t=1}^m\sum\limits_{s=1}^{n}(-1)^{n-s}a^{i_{1}\cdots i_{n-1}j_{s}}_ta_k^{j_{1}\cdots j_{ s-1}j_{s+1}\cdots j_{n}t}=0,
\end{equation}

\begin{equation}\label{eq:biset5}
\sum_{l=1}^{m}c_{i_{1}\cdots i_{n}}^{l}a_{l}^{j_{1}\cdots j_{n}}
=\sum\limits_{s=1}^{n}\sum\limits_{k=1}^n(-1)^{n-k}\sum\limits_{r=1}^ma^{j_1\cdots j_{s-1}r j_{s+1}\cdots j_n}_{i_k}\Big(\sum\limits_{t=1}^s(-1)^{s-t}c_{i_1\cdots \widehat{i_k}\cdots i_n r}^{j_t}+\sum\limits_{t=1}^{n-s}(-1)^{t}c_{i_1\cdots \widehat{i_k}\cdots i_n r}^{j_{s+t}}\Big).
\end{equation}

For the case $n=2$, we have
\begin{equation}\label{eq:2exten}
\sum_{l=1}^{m}c_{i_{1}i_{2}}^{l}a_{l}^{j_{1}j_{2}}=\sum\limits_{r=1}^m\Big( a^{r j_{2}}_{i_1}\Big(-c_{i_2 r}^{j_1}+c_{ i_2 r}^{j_{2}}\Big)+a^{j_1r}_{i_1}\Big(c_{i_2 r}^{j_1}-c_{i_2 r}^{j_2}\Big)+a^{r j_{2}}_{i_2}\Big(c_{i_1 r}^{j_1}-c_{i_1r}^{j_{2}}\Big)
+a^{j_1r}_{i_2}\Big(-c_{i_1r}^{j_1}+c_{i_1r}^{j_2}\Big)\Big).
\end{equation}

\end{theorem}

\begin{proof} By Theorem \ref{thm:Lcons} and Theorem \ref{thm:LandC}, $(L, \mu)$ is an $n$-Lie algebra and $(L, \Delta)$ is an $n$-Lie coalgebra if and only if Eq. \eqref{eq:lset3} and Eq. \eqref{eq:lset4} hold, respectively.
For all $1\leq i_1<\cdots < i_n\leq m$,

$$\Delta\mu(x_{i_{1}},x_{i_{2}},\cdots,x_{i_{n}})=\sum\limits_{l=1}^{m}c_{i_{1}\cdots i_{n}}^{l}\Delta(x_{l})=\sum_{l=1}^{m}\sum_{ j_{1}<\cdots< j_{n}}c_{i_{1}\cdots i_{n}}^{l}a_{l}^{j_{1}\cdots j_{n}}x_{j_{1}}\wedge\cdots\wedge x_{j_{n}},$$

\vspace{5mm}$
\sum\limits_{s=1}^{n}\sum\limits_{k=1}^n(-1)^{n-k}\rho^{\mu}_s(x_{i_1}, \cdots, \widehat{x_{i_k}}, \cdots, x_{i_n})\Delta(x_{i_k})$
\\
$
=\sum\limits_{s=1}^{n}\sum\limits_{k=1}^n(-1)^{n-k}\sum\limits_{t_1<\cdots< t_n}a^{t_1\cdots t_n}_{i_k}\rho^{\mu}_s(x_{i_1}, \cdots, \widehat{x_{i_k}}, \cdots, x_{i_n})(x_{t_1}\wedge \cdots \wedge x_{t_n})$
\\
$
=\sum\limits_{s=1}^{n}\sum\limits_{k=1}^n(-1)^{n-k}\sum\limits_{t_1<\cdots< t_n}a^{t_1\cdots t_n}_{i_k}(x_{t_1}\wedge \cdots \wedge x_{t_{s-1}}\wedge \mu(x_{i_1},\cdots \widehat{x_{i_k}}, \cdots, x_{i_n}, x_{t_s}) \wedge x_{t_{s+1}}\wedge \cdots \wedge x_{t_n})$
\\
$
=\sum\limits_{s=1}^{n}\sum\limits_{k=1}^n(-1)^{n-k}\sum\limits_{t_1<\cdots< t_n}\sum\limits_{l=1}^m c_{i_1\cdots \widehat{i_k}\cdots i_n t_s}^la^{t_1\cdots t_n}_{i_k}(x_{t_1}\wedge \cdots \wedge x_{t_{s-1}}\wedge x_l\wedge x_{t_{s+1}} \wedge \cdots \wedge x_{t_n}).
$

\vspace{2mm}Therefore, for all $1\leq j_1< \cdots <j_n\leq m,$ and $1\leq i_1< \cdots <i_n\leq m,$

$$
\sum_{l=1}^{m}c_{i_{1}\cdots i_{n}}^{l}a_{l}^{j_{1}\cdots j_{n}}
=\sum\limits_{s=1}^{n}\sum\limits_{k=1}^n(-1)^{n-k}\sum\limits_{r=1}^ma^{j_1\cdots j_{s-1}r j_{s+1}\cdots j_n}_{i_k}\Big(\sum\limits_{t=1}^s(-1)^{s-t}c_{i_1\cdots \widehat{i_k}\cdots i_n r}^{j_t}+\sum\limits_{t=1}^{n-s}(-1)^{t}c_{i_1\cdots \widehat{i_k}\cdots i_n r}^{j_{s+t}}\Big).
$$

For the case $n=2$, we have

$$
\sum_{l=1}^{m}c_{i_{1}i_{2}}^{l}a_{l}^{j_{1}j_{2}}
=\sum\limits_{s=1}^{2}\sum\limits_{k=1}^2(-1)^{2-k}\sum\limits_{r=1}^ma^{j_1\cdots j_{s-1}r j_{s+1}\cdots j_2}_{i_k}\Big(\sum\limits_{t=1}^s(-1)^{s-t}c_{i_1\cdots \widehat{i_k}\cdots i_2 r}^{j_t}+\sum\limits_{t=1}^{2-s}(-1)^{t}c_{i_1\cdots \widehat{i_k}\cdots i_2 r}^{j_{s+t}}\Big).
$$
$$
=\sum\limits_{k=1}^2(-1)^{2-k} \sum\limits_{r=1}^m\Big(a^{r j_{2}}_{i_k}\Big(c_{i_1\cdots \widehat{i_k}\cdots i_2 r}^{j_1}+(-1)^{1}c_{i_1\cdots \widehat{i_k}\cdots i_2 r}^{j_{2}}\Big)+a^{j_1r}_{i_k}\Big(-c_{i_1\cdots \widehat{i_k}\cdots i_2 r}^{j_1}+c_{i_1\cdots \widehat{i_k}\cdots i_2 r}^{j_2}\Big)\Big)
$$
$$
=-\sum\limits_{r=1}^m\Big( a^{r j_{2}}_{i_1}\Big(c_{i_2 r}^{j_1}-c_{ i_2 r}^{j_{2}}\Big)+a^{j_1r}_{i_1}\Big(-c_{i_2 r}^{j_1}+c_{i_2 r}^{j_2}\Big)\Big)+\sum\limits_{r=1}^ma^{r j_{2}}_{i_2}\Big(\Big(c_{i_1 r}^{j_1}-c_{i_1r}^{j_{2}}\Big)
+a^{j_1r}_{i_2}\Big(-c_{i_1r}^{j_1}+c_{i_1r}^{j_2}\Big)\Big)
$$
$$
=\sum\limits_{r=1}^m\Big( a^{r j_{2}}_{i_1}\Big(-c_{i_2 r}^{j_1}+c_{ i_2 r}^{j_{2}}\Big)+a^{j_1r}_{i_1}\Big(c_{i_2 r}^{j_1}-c_{i_2 r}^{j_2}\Big)+a^{r j_{2}}_{i_2}\Big(c_{i_1 r}^{j_1}-c_{i_1r}^{j_{2}}\Big)
+a^{j_1r}_{i_2}\Big(-c_{i_1r}^{j_1}+c_{i_1r}^{j_2}\Big)\Big).
$$

The proof is complete.

\end{proof}

\begin{theorem}\label{thm:cobialg} Let $(L,\mu,\Delta)$ be an $n$-Lie bialgebra. Then the triple $(L^{\ast},\Delta^{\ast},\mu^{\ast})$ is an $n$-Lie bialgebra,
and it is called the dual $n$-Lie bialgebra of $(L, \mu, \Delta).$
\end{theorem}

\begin{proof} Since $(L,\mu, \Delta)$ is an $n$-Lie bialgebra, by Theorem \ref{thm:LandC}, $(L^{\ast},\Delta^{\ast})$ is an $n$-Lie algebra
and $(L^{\ast},\mu^{\ast})$ is an $n$-Lie coalgebra, where $\Delta^{\ast}: L^{*\otimes n} \rightarrow L^*$, $\mu^{\ast}: L^* \rightarrow L^{*\otimes n}$.

We need to prove that $\mu^{\ast}:L^{\ast}\rightarrow L^{\ast\otimes n}$ satisfies identity \eqref{eq:bialg1},  that is,
for all $f_{1},\cdots,f_{n}\in L^{\ast},$
\begin{equation}\label{eq:cobialg}
\mu^{\ast}(\Delta^{\ast}(f_{1},\cdots,f_{n}))=\sum_{s=1}^{n}\sum_{k=1}^{n}(-1)^{n-k}\rho_{s}^{\Delta^{\ast}}(f_{1},\cdots,\widehat{f_{k}},\cdots f_{n})\mu^{\ast}(f_{k}),
\end{equation}
where $\rho_{s}^{\Delta^{\ast}}(f_1, \cdots, \widehat{f_k}, \cdots f_n)=\underbrace{1\otimes \cdots\otimes 1}_{s-1}\otimes ad_{\Delta^*}(f_1, \cdots, \widehat{f_k}, \cdots f_n)\otimes \underbrace{1\otimes \cdots\otimes 1}_{n-s}: A^{* \otimes n}\rightarrow A^{* \otimes n}.$

For all $x_{1},x_{2},\cdots,x_{n}\in L,$

$\langle \mu^{\ast}(\Delta^{\ast}(f_{1},\cdots,f_{n})),x_{1}\otimes\cdots\otimes x_{n}\rangle$
$=\langle\Delta^{\ast}(f_{1},\cdots,f_{n}),\mu(x_{1},x_{2},\cdots, x_{n})\rangle$\\
$=\langle f_{1}\otimes\cdots\otimes f_{n},\Delta\mu(x_{1},x_{2},\cdots, x_{n})\rangle$\\
$=\langle f_{1}\otimes\cdots\otimes f_{n},\sum\limits_{s=1}^{n}\sum\limits_{k=1}^{n}(-1)^{n-k}\rho_{s}^{\mu}(x_{1},\cdots,\widehat{x_{k}},\cdots x_{n})\Delta(x_{k})\rangle.
$

Suppose that $\Delta(x_{1})=\sum\limits_{1_j\cdots n_j}y_{1_{j}}^{1}\otimes\cdots\otimes y_{n_{j}}^{1},$ where
$y_{1_{j}}^{1},y_{2_{j}}^{1},\cdots,y_{n_{j}}^{1}\in L.$ Then

\begin{align*}
&\langle f_{1}\otimes\cdots\otimes f_{n},\quad\sum_{s=1}^{n}(-1)^{n-1}\rho_{s}^{\mu}(x_{2},\cdots,x_{n})\Delta(x_{1})\rangle\\
=&\langle f_{1}\otimes\cdots\otimes f_{n},\quad\sum_{s=1}^{n}(-1)^{n-1}\rho_{s}^{\mu}(x_{2},\cdots,x_{n})\Big(\sum\limits_{1_j\cdots n_j}y_{1_{j}}^{1}\otimes\cdots\otimes y_{n_{j}}^{1}\Big)\rangle\\
=&\langle f_{1}\otimes\cdots\otimes f_{n},\quad\sum_{s=1}^{n}(-1)^{n-1}\sum\limits_{1_j\cdots n_j}y_{1_{j}}^{1}\otimes \cdots\otimes ad_{\mu}(x_{2},\cdots,x_{n})(y_{s_{j}}^{1})\otimes\cdots\otimes y_{n_{j}}^{1}\rangle\\
=& \sum_{s=1}^{n}(-1)^{n-1}\sum\limits_{1_j\cdots n_j}\quad\langle f_{1}\otimes\cdots\otimes f_{n}, \quad y_{1_{j}}^{1}\otimes \cdots\otimes ad_{\mu}(x_{2},\cdots,x_{n})(y_{s_{j}}^{1})\otimes\cdots\otimes y_{n_{j}}^{1}\rangle\\
=&- \sum_{s=1}^{n}(-1)^{n-1}\sum\limits_{1_j\cdots n_j}\quad\langle f_{1}\otimes\cdots \otimes ad^*_{\mu}(x_{2},\cdots,x_{n})(f_s)\otimes \cdots \otimes f_{n}, \quad y_{1_{j}}^{1}\otimes \cdots\otimes y_{n_{j}}^{1}\rangle\\
=&-(-1)^{n-1}\sum_{s=1}^{n}\quad\langle f_{1}\otimes\cdots\otimes ad_{\mu}^{\ast}(x_{2},\cdots,x_{n})(f_{s})\otimes
\cdots\otimes f_{n}, \quad \Delta(x_{1})\rangle\\
=&-(-1)^{n-1}\sum_{s=1}^{n}\quad\langle \Delta^*(f_{1}, \cdots, ad_{\mu}^{\ast}(x_{2},\cdots,x_{n})(f_{s}),
\cdots, f_{n}), \quad x_{1}\rangle\\
=&-(-1)^{n-1}\sum_{s=1}^{n}(-1)^{n-s}\quad\langle ad_{\Delta^*}(f_{1},\cdots, \widehat{f_s}, \cdots, f_n)ad_{\mu}^{\ast}(x_{2},\cdots,x_{n})(f_{s})
, \quad x_{1}\rangle\\
=&-\sum_{s=1}^{n}(-1)^{-s-1}\quad\langle ad_{\Delta^*}(f_{1},\cdots, \widehat{f_s}, \cdots, f_n)ad_{\mu}^{\ast}(\widehat{x_1}, x_{2},\cdots,x_{n})(f_{s})
, \quad x_{1}\rangle.\\
\end{align*}

Then we have

$\langle f_{1}\otimes\cdots\otimes f_{n},\sum\limits_{s=1}^{n}\sum\limits_{k=1}^{n}(-1)^{n-k}\rho_{s}^{\mu}(x_{1},\cdots,\widehat{x_{k}},\cdots x_{n})\Delta(x_{k})\rangle$
\\$
=-\sum\limits_{k=1}^{n}\sum\limits_{s=1}^{n}(-1)^{-s-k}\quad\langle ad_{\Delta^*}(f_{1},\cdots, \widehat{f_s}, \cdots, f_n)ad_{\mu}^{\ast}(x_{1},\cdots,\widehat{x_{k}},\cdots x_{n})(f_{s}), \quad x_{k}\rangle$
\\
$=\sum\limits_{k=1}^{n}\sum\limits_{s=1}^{n}(-1)^{-s-k}\quad\langle ad_{\mu}^{\ast}(x_{1},\cdots,\widehat{x_{k}},\cdots x_{n})(f_{s}),\quad ad_{\Delta^*}^{\ast}(f_{1},\cdots, \widehat{f_s}, \cdots, f_n)x_{k}\rangle$
\\
$
=-\sum\limits_{k=1}^{n}\sum\limits_{s=1}^{n}(-1)^{-s-k}\quad\langle f_{s}, \quad ad_{\mu}(x_{1},\cdots,\widehat{x_{k}},\cdots x_{n})ad_{\Delta^*}^{\ast}(f_{1},\cdots, \widehat{f_s}, \cdots, f_n)x_{k}\rangle$
\\
$
=-\sum\limits_{k=1}^{n}\sum\limits_{s=1}^{n}(-1)^{-s-k}\quad\langle f_{s}, \quad \mu(x_{1},\cdots,\widehat{x_{k}},\cdots x_{n}, ad_{\Delta^*}^{\ast}(f_{1},\cdots, \widehat{f_s}, \cdots, f_n)x_{k})\rangle$
\\
$
=-\sum\limits_{k=1}^{n}\sum\limits_{s=1}^{n}(-1)^{-s-k}(-1)^{n-k}\quad\langle f_{s}, \quad \mu(x_{1},\cdots, ad_{\Delta^*}^{\ast}(f_{1},\cdots, \widehat{f_s}, \cdots, f_n)x_{k}, \cdots x_{n})\rangle$\\
$
=-\sum\limits_{k=1}^{n}\sum\limits_{s=1}^{n}(-1)^{n-s}\quad\langle \mu^*(f_{s}), \quad x_{1}\otimes\cdots\otimes  ad_{\Delta^*}^{\ast}(f_{1},\cdots, \widehat{f_s}, \cdots, f_n)x_{k}\otimes\cdots\otimes x_{n})\rangle$
\\
$=\sum\limits_{k=1}^{n}\sum\limits_{s=1}^{n}(-1)^{n-s}\langle \rho_k^{\Delta^*}(f_{1},\cdots, \widehat{f_s}, \cdots, f_n)\mu^{\ast}(f_{s}),\quad x_{1}\otimes\cdots\otimes x_{k}\otimes\cdots\otimes x_{n}\rangle.$

The identity \eqref{eq:cobialg} holds. The proof is complete.
\end{proof}

\begin{coro}\label{cor:1}
Let $ L $ be an $m$-dimensional  vector space over a field $\mathbb F$ with a basis $ x_1, \cdots, x_m $, \quad $\mu: L^{\wedge n}\rightarrow L$ \quad and \quad $\Delta: L \rightarrow L^{\wedge n}$ be linear maps, and suppose
$$
\mu(x_{i_{1}},x_{i_{2}},\cdots,x_{i_{n}})=\sum_{l=1}^{m}c_{i_{1}\cdots i_{n}}^{l}x_{l}, \quad \Delta(x_{l})=\sum_{i_1 \cdots i_n}a_{l}^{i_{1}\cdots i_{n}}x_{i_{1}}\otimes \cdots\otimes  x_{i_{n}},
$$
where $c_{i_{1}\cdots i_{n}}^{l}, a_{l}^{j_{1}\cdots j_{n}}\in \mathbb F,$ $1\leq i_{1}, \cdots, i_{n}, j_1, \cdots, j_n\leq m.$
and constants $$\Big\{~~ c_{i_1\cdots i_n}^l, \quad a_{l}^{i_{1}\cdots i_{n}}~| ~ 1\leq i_1, \cdots i_n, l\leq m ~~ \Big\} $$
satisfy Eq. \eqref{eq:lset3} and Eq. \eqref{eq:lset4}. Then Eq.\eqref{eq:biset5} holds if and only if the following identity holds
\begin{equation}\label{eq:biset6}
\sum_{l=1}^{m}a^{i_{1}\cdots i_{n}}_{l}c^{l}_{j_{1}\cdots j_{n}}
=\sum\limits_{s=1}^{n}\sum\limits_{k=1}^n(-1)^{n-k}\sum\limits_{r=1}^mc_{j_1\cdots j_{s-}rj_{s+1}\cdots j_n}^{i_k}\Big(\sum\limits_{t=1}^s(-1)^{s-t}a^{i_1\cdots \widehat{i_k}\cdots i_n r}_{j_t}+\sum\limits_{t=1}^{n-s}(-1)^{t}a^{i_1\cdots \widehat{i_k}\cdots i_n r}_{j_{s+t}}\Big).
\end{equation}

\end{coro}

\begin{proof}

The result follows from Theorem \ref{thm:bialgconst} and Theorem \ref{thm:cobialg} directly.

\end{proof}

\begin{defn}\label{defn:biequlent} Two $n$-Lie bialgebras $(L_{1},\mu_{1},\Delta_{1})$ and  $(L_{2}, \mu_{2}, \Delta_{2})$ are called equivalent if there exists linear  isomorphism $f: L_{1}\rightarrow L_{2}$ such that

$(1)$ $f: (L_{1},\mu_{1})\rightarrow (L_{2},\mu_{2})$ is an $n$-Lie algebra isomorphism, that is,  ~ for all  $ x_{1}, x_{2}, \cdots, x_{n}\in L_{1},$

$$f\mu_{1}(x_{1},x_{2},\cdots,x_{n})=\mu_{2}(f(x_{1}),f(x_{2}),\cdots,f(x_{n})).$$

$(2)$ $f: ( L_{1},\Delta_{1}))\rightarrow (L_{2}, \Delta_{2})$ is an $n$-Lie coalgebra isomorphism, that is, ~~ for all  $x\in L_{1},$

$$(\underbrace{f\otimes\cdots\otimes f}_{n})\Delta_{1}(x)=\Delta_{2}(f(x)). $$
\end{defn}

For a given $n$-Lie algebra $(L, \mu),$ in order to find all the $n$-Lie bialgebra structures on $L,$ we should find all the $n$-Lie coalgebra structures
on $L$ which  are compatible with the $n$-Lie algebra $L.$ Although  $n$-Lie coalgebra  $(L, \Delta_1)$ is isomorphic to $n$-Lie coalgebra  $(L, \Delta_1)$, but it may leads to a different $n$-Lie bialgebra structures on $(L, \mu)$, that is, the $n$-Lie bialgebras   $(L, \mu, \Delta_1)$  and
$(L, \mu, \Delta_2)$ may not be equivalent.

\begin{exam}
Let $(L, \mu)$ be an $(n+1)$-dimensional $n$-Lie algebra with a basis $x_{1},x_{2},\cdots,x_{n+1},$ $(L, \Delta_i)$ for $i=1, 2, 3$ be $n$-Lie coalgebras, where  the multiplication
 $\mu:  L^{\wedge n} \rightarrow L$,  and $\Delta_i: L\rightarrow L^{\wedge n} $ are as follows: $ \left\{
\begin{array}{l}
\mu(x_{2},x_{3},\cdots,x_{n+1})=x_{1},\\
\mu(x_{1},x_{3},\cdots,x_{n+1})=x_{2}.
\end{array}
\right.
$

\vspace{2mm}$\left\{
\begin{array}{l}
\Delta_{1}(x_{1})=x_{1}\wedge x_{3}\wedge x_{4}\wedge\cdots\wedge x_{n+1},\\
\Delta_{1}(x_{2})=x_{2}\wedge x_{3}\wedge x_{4}\wedge\cdots\wedge x_{n+1},\\
\Delta_{1}(x_{j})=0,j\geq 3;
\end{array}
\right.
$

\vspace{2mm}$\left\{
\begin{array}{l}
\Delta_{2}(x_{1})=x_{1}\wedge x_{2}\wedge x_{4}\wedge\cdots\wedge x_{n+1},\\
\Delta_{2}(x_{3})=x_{3}\wedge x_{2}\wedge x_{4}\wedge\cdots\wedge x_{n+1},\\
\Delta_{2}(x_{j})=0,j\neq 1,3;
\end{array}
\right.
$

\vspace{2mm}$\left\{
\begin{array}{l}
\Delta_{3}(x_{2})=x_{2}\wedge x_{1}\wedge x_{4}\wedge\cdots\wedge x_{n+1},\\
\Delta_{3}(x_{3})=x_{3}\wedge x_{1}\wedge x_{4}\wedge\cdots\wedge x_{n+1},\\
\Delta_{3}(x_{j})=0,j\neq 2,3.
\end{array}
\right.
$

\vspace{2mm}Define linear isomorphism  $\varphi_{12},\varphi_{13},\varphi_{23}:L\rightarrow L$ by
$$\varphi_{12}(x_{1})=x_{1},\varphi_{12}(x_{2})=x_{3},\varphi_{12}(x_{3})=x_{2},\varphi_{12}(x_{i})=x_{i}, \quad 4\leq i\leq n+1$$
$$\varphi_{13}(x_{1})=x_{2},\varphi_{13}(x_{2})=x_{3},\varphi_{13}(x_{3})=x_{1},\varphi_{13}(x_{i})=x_{i}, \quad 4\leq i\leq n+1$$
$$\varphi_{23}(x_{1})=x_{2},\varphi_{23}(x_{2})=x_{1},\varphi_{23}(x_{3})=x_{3},\varphi_{23}(x_{i})=x_{i}, \quad 4\leq i\leq n+1$$

By a direct computation, we obtain that $\varphi_{12}: (L,\Delta_{1})\rightarrow (L,\Delta_{2})$,
\quad  $\varphi_{13}: (L,\Delta_{1})\rightarrow (L,\Delta_{3})$, and $\varphi_{23}: (L,\Delta_{2})\rightarrow (L,\Delta_{3})$ are $n$-Lie coalgebra isomorphism, respectively.
And $ (L, \mu, \Delta_{1})$, $ (L, \mu, \Delta_{2})$,  and $ (L, \mu, \Delta_{3})$ are $n$-Lie bialgebras.

Let $\sigma: L \rightarrow L$ be a linear isomorphism, and assume

$$\sigma(x_i)=\sum\limits_{j=1}^{n+1}a_{ij}x_j, ~~a_{ij}\in \mathbb F, \quad 1\leq i\leq n+1.$$
Then  $\sigma$ is an automorphism of $n$-Lie algebra $(L, \mu)$ if and only if $\sigma$ satisfies that

$$\sigma\left(
      \begin{array}{cccccccccccc}
      x_1 \\
      x_2  \\
     x_3  \\
      x_4  \\
      \vdots  \\
      x_{(n+1)}  \\
      \end{array}
    \right)
=\left(
      \begin{array}{cccccccccccc}
      a_{11} & a_{12} & 0 & 0  & \cdots & 0 \\
      da_{12} & da_{11} & 0  & 0 & \cdots  & 0 \\
      a_{31} & a_{32} & a_{33} & a_{34}  &  \cdots & a_{3(n+1)} \\
       a_{41} & a_{42} & a_{43} & a_{44}  &  \cdots & a_{4(n+1)} \\
       \vdots & \vdots & \vdots & \vdots & \vdots  & \vdots \\
       a_{(n+1)1} & a_{(n+1)2} & a_{(n+1)3} & a_{(n+1)4} &  \cdots  & a_{(n+1)(n+1)} \\
      \end{array}
    \right)\left(
      \begin{array}{cccccccccccc}
      x_1 \\
      x_2  \\
     x_3  \\
      x_4  \\
      \vdots  \\
      x_{(n+1)}  \\
      \end{array}
    \right),$$
that is,
$a_{1k}=a_{2k}=0, ~~~  3\leq k\leq n+1;$ \quad $a_{11}=da_{22}, ~~ a_{12}=d a_{21}, $ where $a_{11}^2-a_{12}^2\neq 0,$
$$ d=\begin{vmatrix}
a_{3 3} &\cdots & a_{3 n+1}\\
\vdots &      & \vdots\\
a_{n+1 3} &\cdots & a_{n+1 n+1}\\
\end{vmatrix}=1 ~~ \text{or} ~~ -1.$$

In the case
$$
\sigma(x_1)=x_2, \sigma(x_2)=x_1, \sigma(x_l)=x_l, \quad \forall 3\leq l\leq n+1,
$$
(that is, $a_{12}= a_{21}=a_{ll}=1, $ for $3\leq l\leq n+1$, and others are zero), $\sigma$ is an isomorphism of $n$-Lie coalgebra $(L, \Delta_2)$ and  $(L, \Delta_3)$. Therefore, $n$-Lie bialgebras   $ (L, \mu, \Delta_2)$ and $ (L, \mu, \Delta_3)$ are equivalent.

We will prove that $n$-Lie bialgebras $(L, \mu, \Delta_{2})$ and $(L, \mu, \Delta_{1})$ are non-equivalent.

In fact,  if  $n$-Lie bialgebras $(L, \mu, \Delta_{2})$ and $(L, \mu, \Delta_{1})$ are equivalent, then $\sigma$ is an automorphism of
$n$-Lie algebra $(L, \mu)$, and for all $k\geq 3$, we have

 $\Delta_2(\sigma(x_k))=\Delta_2(\sum\limits_{j=1}^{n+1}a_{kj}x_j)=a_{k1}\Delta_2(x_1)+a_{k3}\Delta_2(x_3)$
 \\
 $=a_{31} x_{1}\wedge x_{2}\wedge x_{4}\wedge\cdots\wedge x_{n+1}+a_{33} x_{3}\wedge x_{2}\wedge x_{4}\wedge\cdots\wedge x_{n+1}
 $
 \\
 $=\sigma^n(\Delta_1(x_3))=0.$

\vspace{2mm}\noindent We obtain that $a_{k1}=a_{k3}=0$ for all $k\geq 3$, therefore, $ d=\begin{vmatrix}
a_{3 3} &\cdots & a_{3 n+1}\\
\vdots &      & \vdots\\
a_{n+1 3} &\cdots & a_{n+1 n+1}\\
\end{vmatrix}=0.$ This contradicts the invert of the automorphism $\sigma$.

Therefore, for all $n$-Lie algebra automorphism $\sigma$ of $(L, \mu)$,  $\sigma: (L, \Delta_1) \rightarrow (L, \Delta_2)$ is not a isomorphism of $n$-Lie coalgebra.

\end{exam}

\section{ Two -dimensional extension of   $n$-Lie bialgebras}
\mlabel{sec:bialg}

In this section, we construct $(n+1)$-Lie bialgebras from a known $n$-Lie bialgebra by two dimensional extensions \cite{BWJ1}.

Let $(L, \mu)$ be an $n$-Lie algebra with a basis $x_1, \cdots, x_m$, $B: L\otimes L\rightarrow \mathbb F$ be a non-degenerate bilinear  symmetric
function satisfying $ad_{\mu}$-invariant, that is, for all $y_1, \cdots, y_{n-1}, , x, z\in L, $
\begin{equation}\label{eq:metric}
B(\mu(y_1, \cdots, y_{n-1}, x), z)+B(x, \mu(y_1, \cdots, y_{n-1}, z))=0.
\end{equation}
The triple $(L, \mu, B)$ is called a metric $n$-Lie algebra \cite{BWL1}.

\begin{lemma} \cite{BWJ1}\label{lem:Lexten}
 Let  $(L,\mu)$ be an $m$-dimensional
$n$-Lie algebra with a basis $x_1,$ $ \cdots,$ $ x_m$. $B: L\otimes L\rightarrow \mathbb F$ be a non-degenerate bilinear  symmetric
function satisfying $ad_{\mu}$-invariant. Then
 $(\bar{L}=L\oplus \mathbb F x_0\oplus
\mathbb F x_{-1}, \bar{\mu})$ is an $(n+1)$-Lie algebra,
where $\bar{\mu}$ is  defined as follows,  for all   $1\leq k\leq n+1$,

\begin{equation}\label{eq:ex}
 \bar{\mu}(x_{i_1}, ~\ldots,\underbrace{ x_0}_k, \ldots,
~x_{i_n})=(-1)^{k-1}\mu(x_{i_1}, \ldots, x_{i_n}), ~1\leq i_1,
\ldots, i_{n}\leq m;
\end{equation}

\begin{equation}\label{eq:ext}
\bar{\mu}(x_{i_1}, ~\ldots,\underbrace{ x_{-1}}_k, \ldots, ~x_{i_n})=0,
~~ ~ 0\leq  i_1, \ldots, i_n\leq m;
\end{equation}

\begin{equation}\label{eq:exte}
\bar{\mu}(x_{i_1}, ~ \ldots, ~ x_{i_n}, ~ x_{i_{n+1}})=B(\mu(x_{i_1},
\ldots, x_{i_n}), x_{i_{n+1}})x_{-1}, ~1\leq i_1, \ldots,
i_{n+1}\leq m.
\end{equation}

\vspace{3mm} $\bar{B}: \bar{L}\otimes \bar{L} \rightarrow
\mathbb F$ defined as:  for all $x, y\in L$,

\begin{equation}\label{eq:exten}
 \bar{B}(x, y)=B(x, y),   ~ \bar{B}(x_0,
x_0)=1, \bar{B}(x_{-1}, x_0)=(-1)^{n-1},
\end{equation}
$$\bar{B}(x_{-1}, y)=\bar{B}(x_{-1}, x_{-1})=\bar{B}(x_0, y)=0,$$
is ad$_{\bar\mu}$-invariant, and $\bar B$ is non-degenerate if and only if $B$ is non-degenerate.

\end{lemma}

\begin{lemma}\label{lem:exten11}
 Let  $(L,\mu)$ be an $m$-dimensional
$n$-Lie algebra with a basis $x_1,$ $ \cdots,$ $ x_m$. Then
 $(\bar{L}=L\oplus \mathbb F x_0\oplus
\mathbb F x_{-1}, \bar{\mu})$ is an $(n+1)$-Lie algebra,
where $\bar{\mu}$ is  defined as  for all  $k$, $1\leq k\leq n+1$,

\begin{equation}\label{eq:2exten2} \bar{\mu}(x_{i_1}, ~\ldots,\underbrace{ x_0}_k, \ldots,
~x_{i_n})=(-1)^{k-1}\mu(x_{i_1}, \ldots, x_{i_n}), ~1\leq i_1,
\ldots, i_{n}\leq m;
\end{equation}

\begin{equation}\label{eq:3exten3}
\bar{\mu}(x_{i_1}, ~\ldots,\underbrace{ x_{-1}}_k, \ldots, ~x_{i_n})=0,
~~ ~ 0\leq  i_1, \ldots, i_n\leq m;
\end{equation}

\begin{equation}\label{eq:4exten4}
\bar{\mu}(x_{i_1}, ~ \ldots, ~ x_{i_n}, ~ x_{i_{n+1}})=0, ~1\leq i_1, \ldots,
i_{n+1}\leq m.
\end{equation}

\end{lemma}

\begin{proof}

The result follows from a direct computation.
\end{proof}

\begin{theorem}\label{thm:biext}
Let $(L, \mu, \Delta)$ be an $m$-dimensional  $n$-Lie bialgebra, elements $x_0, x_{-1}$ be not contained in $L$, $B: L\otimes L \rightarrow \mathbb F$ be a bilinear symmetric function  satisfying ad$_{\mu}$-invariant. Then $(\bar L, \bar{\mu}, \bar \Delta)$ is an $(m+2)$-dimensional $(n+1)$-Lie bialgebra, where
$\bar L=L\oplus \mathbb F x_0\oplus \mathbb F x_{-1}, $ $\bar{\mu}$ is defined as Eqs. \eqref{eq:ex}, \eqref{eq:ext} and \eqref{eq:exte}, and
$(\bar \Delta)^*$ is defined as Eqs. \eqref{eq:2exten2}, \eqref{eq:3exten3} and \eqref{eq:4exten4}.

\end{theorem}

\begin{proof}
Let $\{ x_{-1}, x_0, x_1, \cdots, x_m\}$ be a basis of $\bar L$, and  $\{ x^{-1}, x^0, x^1, \cdots, x^m\}$ be the dual basis of $(\bar L)^*$.
By Lemma \ref{lem:exten11} and Eq. \eqref{eq:CCL}, we   can suppose

\begin{equation}\label{eq:bardelt}
\bar \Delta (x_0)= \bar \Delta(x_{-1})=0, \quad  \bar{\Delta}(x_k)=\sum\limits_{j_1\cdots j_n}a^{j_1\cdots j_n}_kx_{-1}\wedge x_{j_1}\wedge\cdots\wedge x_{j_n}, 1\leq j_1<\cdots <j_n\leq m.
\end{equation}

Then we need to prove $\bar \mu$ and $\bar{\Delta}$ satisfying  Eq. \eqref{eq:bialg1}.  So we  dived it into four cases.

\vspace{2mm} {\bf (1)}. Let  $ j_1=-1,$ for all $  1\leq j_2, \cdots, j_{n+1}\leq m,$ by Lemma \ref{lem:Lexten}, we have

\vspace{3mm}
$ \bar{\Delta}\bar{\mu}(x_{-1}, x_{j_2}, \cdots, x_{j_{n+1}})=\bar{\Delta}(0)=0.$

\vspace{2mm}
$\sum\limits_{s=1}^{n+1}\sum\limits_{k=1}^{n+1}(-1)^{n+1-k}\rho^{\bar{\mu}}_s(x_{j_1},\cdots, \widehat{x_{j_k}}, \cdots, x_{j_{n+1}})\bar{\Delta}(x_{j_k})
$
\\
$
=(-1)^{n}\sum\limits_{s=1}^{n+1}\rho_s^{\bar{\mu}}(x_{j_2}, \cdots, x_{j_{n+1}})\bar{\Delta}(x_{-1})
$
$+\sum\limits_{s=1}^{n+1}\sum\limits_{k=2}^{n+1}(-1)^{n+1-k}\rho^{\bar{\mu}}_s(x_{-1}, x_{j_2},\cdots, \widehat{x_{j_k}}, \cdots, x_{j_{n+1}})\bar{\Delta}(x_{j_k})
$
\\
$
=\sum\limits_{s=1}^{n+1}\sum\limits_{k=2}^{n+1}(-1)^{n+1-k}\rho^{\bar{\mu}}_s(x_{-1}, x_{j_2},\cdots, \widehat{x_{j_k}}, \cdots, x_{j_{n+1}})\sum\limits_{1\leq l_1\cdots l_n\leq m}a^{l_1 \cdots l_n}_{j_k}
x_{-1}\wedge x_{l_1}\wedge \cdots \wedge x_{l_n}.
$

\vspace{2mm}Since for all $-1\leq t_1, \cdots, t_{n+1}\leq m,$

$\rho^{\bar{\mu}}_s(x_{-1}, x_{j_2},\cdots, \widehat{x_{j_k}}, \cdots, x_{j_{n+1}})(x_{t_1}\wedge \cdots \wedge x_{t_{n+1}})$
\\
$
=x_{t_1}\wedge\cdots \wedge x_{t_{s-1}}\wedge \bar{\mu}(x_{-1}, x_{j_2},\cdots, \widehat{x_{j_k}}, \cdots, x_{j_{n+1}}, x_{t_s})\wedge x_{t_{s+1}}\wedge \cdots x_{t_{n+1}})=0,
$

We have
$\sum\limits_{s=1}^{n+1}\sum\limits_{k=1}^{n+1}(-1)^{n+1-k}\rho^{\bar{\mu}}_s(x_{j_1},\cdots, \widehat{x_{j_k}}, \cdots, x_{j_{n+1}})\bar{\Delta}(x_{j_k})
$
$ =\bar{\Delta}\bar{\mu}(x_{-1}, x_{j_2}, \cdots, x_{j_{n+1}})=0.$

\vspace{2mm}{\bf (2)}. For $ j_1=-1,\quad j_2=0, \quad 1\leq j_3, \cdots, j_{n+1}\leq m,$
by Eq.\eqref{eq:ext},   $\bar{\Delta}(x_{-1})=\bar{\Delta}(x_0)=0$, then we have

\vspace{2mm}
$\sum\limits_{s=1}^{n+1}\sum\limits_{k=1}^{n+1}(-1)^{n+1-k}\rho^{\bar{\mu}}_s(x_{j_1},\cdots, \widehat{x_{j_k}}, \cdots, x_{j_{n+1}})\bar{\Delta}(x_{j_k})
$
\\
$
=(-1)^{n}\sum\limits_{s=1}^{n+1}\rho_s^{\bar{\mu}}(x_{j_2}, \cdots, x_{j_{n+1}})\bar{\Delta}(x_{-1})
$
$+(-1)^{n-1}\sum\limits_{s=1}^{n+1}\rho_s^{\bar{\mu}}(x_{-1}, x_{j_3}, \cdots, x_{j_{n+1}})\bar{\Delta}(x_{0})
$
\\
$+\sum\limits_{s=1}^{n+1}\sum\limits_{k=3}^{n+1}(-1)^{n+1-k}\rho^{\bar{\mu}}_s(x_{-1}, x_0, x_{j_3},\cdots, \widehat{x_{j_k}}, \cdots, x_{j_{n+1}})\bar{\Delta}(x_{j_k})
$
\\
$=\bar{\Delta}\bar{\mu}(x_{-1}, x_{0}, x_{j_3}\cdots, x_{j_{n+1}})$.

\vspace{2mm}{\bf(3)}.  For  $j_1=0,$ and  $ 1\leq j_2, \cdots, j_{n+1}\leq m,$ thanks to Eq. \eqref{eq:ex} and Eq.\eqref{eq:bardelt},

\vspace{3mm}
$ \bar{\Delta}\bar{\mu}(x_{0}, x_{j_2}, \cdots, x_{j_{n+1}})=\bar{\Delta}\mu(x_{j_2}, \cdots, x_{j_{n+1}})=\bar{\Delta}\Big(\sum\limits_{t=1}^mc_{j_2\cdots j_{n+1}}^tx_t\Big)$
$=\sum\limits_{t=1}^{m}c_{j_2\cdots j_{n+1}}^t\bar{\Delta}(x_t)$
\\
$=\sum\limits_{ l_1\cdots l_n}\sum\limits_{t=1}^ma_t^{l_1\cdots l_n} c_{j_2\cdots j_{n+1}}^t x_{-1}\wedge x_{l_1}\wedge\cdots \wedge x_{l_n}$
$=x_{-1}\wedge \Delta\mu(x_{j_2}, \cdots, x_{j_{n+1}})$.

\vspace{3mm}
$\sum\limits_{s=1}^{n+1}\sum\limits_{k=1}^{n+1}(-1)^{n+1-k}\rho^{\bar{\mu}}_s(x_{j_1},\cdots, \widehat{x_{j_k}}, \cdots, x_{j_{n+1}})\bar{\Delta}(x_{j_k})
$
\\
$
=(-1)^{n}\sum\limits_{s=1}^{n+1}\rho_s^{\bar{\mu}}(x_{j_2}, \cdots, x_{j_{n+1}})\bar{\Delta}(x_{0})
$
$+\sum\limits_{s=1}^{n+1}\sum\limits_{k=2}^{n+1}(-1)^{n+1-k}\rho^{\bar{\mu}}_s(x_{0}, x_{j_2},\cdots, \widehat{x_{j_k}}, \cdots, x_{j_{n+1}})\bar{\Delta}(x_{j_k})
$
\\
$
=\sum\limits_{s=1}^{n+1}\sum\limits_{k=2}^{n+1}(-1)^{n+1-k}\rho^{\bar{\mu}}_s(x_{0}, x_{j_2},\cdots, \widehat{x_{j_k}}, \cdots, x_{j_{n+1}})\Big(\sum\limits_{ l_1\cdots l_n}a^{l_1\cdots l_n}_{j_k}x_{-1}\wedge x_{l_1}\wedge\cdots \wedge x_{l_n}\Big)
$
\\
$
=\sum\limits_{s=2}^{n+1}\sum\limits_{k=2}^{n+1}(-1)^{n+1-k}\sum\limits_{l_1\cdots l_n}a^{l_1\cdots l_n}_{j_k}\rho^{\bar{\mu}}_s(x_{0}, x_{j_2},\cdots, \widehat{x_{j_k}}, \cdots, x_{j_{n+1}})(x_{-1}\wedge x_{l_1}\wedge\cdots \wedge x_{l_n})
$
\\
$
=\sum\limits_{s, k=2}^{n+1}(-1)^{n+1-k}\sum\limits_{l_1\cdots l_n}a^{l_1\cdots l_n}_{j_k}
(x_{-1}\wedge\cdots \wedge \bar{\mu}(x_{0}, x_{j_2},\cdots, \widehat{x_{j_k}}, \cdots, x_{j_{n+1}}, x_{l_{s-1}})\wedge
x_{l_{s}}\wedge\cdots\wedge x_{l_n})
$
\\
$
=\sum\limits_{s, k=2}^{n+1}(-1)^{n+1-k}\sum\limits_{l_1\cdots l_n}a^{l_1\cdots l_n}_{j_k}
(x_{-1}\wedge\cdots\wedge \mu( x_{j_2},\cdots, \widehat{x_{j_k}}, \cdots, x_{j_{n+1}}, x_{l_{s-1}})\wedge
x_{l_{s}}\wedge\cdots\wedge x_{l_n})
$
\\
$
=x_{-1}\wedge\Big(\sum\limits_{s, k=2}^{n+1}(-1)^{n+1-k}\rho_s^{\mu}(x_{j_2},\cdots, \widehat{x_{j_k}}, \cdots, x_{j_{n+1}})\Delta(x_{j_k})\Big)
$
\\
$
= \bar{\Delta}\bar{\mu}(x_{0}, x_{j_2}, \cdots, x_{j_{n+1}}),
$ where $1\leq l_1, \cdots, l_n\leq m.$ Therefore,

 $\bar{\Delta}\bar{\mu}(x_{0}, x_{j_2}, \cdots, x_{j_{n+1}})=\sum\limits_{s=1}^{n+1}\sum\limits_{k=1}^{n+1}(-1)^{n+1-k}\rho^{\bar{\mu}}_s(x_{j_1},\cdots, \widehat{x_{j_k}}, \cdots, x_{j_{n+1}})\bar{\Delta}(x_{j_k}).
$

\vspace{3mm} {\bf (4)}.  For all $1\leq j_1, \cdots, j_{n+1}\leq m,$ by  Eq.\eqref{eq:exte},

\vspace{2mm}
$\bar{\Delta}\bar{\mu}(x_{j_1}, \cdots, x_{j_{n+1}})=B\big(\mu(x_{j_1}, \cdots, x_{j_n}), x_{j_{n+1}}\big) \bar{\Delta}(x_{-1})=0.$

\vspace{3mm}
$\sum\limits_{s=1}^{n+1}\sum\limits_{k=1}^{n+1}(-1)^{n+1-k}\rho^{\bar{\mu}}_s(x_{j_1},\cdots, \widehat{x_{j_k}}, \cdots, x_{j_{n+1}})\bar{\Delta}(x_{j_k})
$
\\
 $=\sum\limits_{s=1}^{n+1}\sum\limits_{k=1}^{n+1}(-1)^{n+1-k}\sum\limits_{l_1\cdots l_n}a^{l_1\cdots l_n}_{j_k}\rho^{\bar{\mu}}_s(x_{j_1},\cdots, \widehat{x_{j_k}}, \cdots, x_{j_{n+1}})(x_{-1}\wedge x_{l_1}\wedge\cdots\wedge x_{l_n}),
$
\\where $1\leq l_1, \cdots, l_n\leq m.$

Since $\rho^{\bar{\mu}}_1(x_{j_1},\cdots, \widehat{x_{j_k}}, \cdots, x_{j_{n+1}})(x_{-1}\wedge x_{l_1}\wedge\cdots\wedge x_{l_n})=0,
$  and for $s\geq 2, $

$\rho^{\bar{\mu}}_s(x_{j_1},\cdots, \widehat{x_{j_k}}, \cdots, x_{j_{n+1}})(x_{-1}\wedge x_{l_1}\wedge\cdots\wedge x_{l_n})$
\\$
=x_{-1}\wedge x_{l_1}\wedge\cdots\wedge \bar{\mu}(x_{j_1},\cdots, \widehat{x_{j_k}}, \cdots, x_{j_{n+1}}, x_{l_{s-1}})\wedge x_{l_{s}}\wedge \cdots \wedge x_{l_n}
$
\\
$
=B\big(\mu(x_{j_1},\cdots, \widehat{x_{j_k}}, \cdots, x_{j_{n+1}}, x_{l_s}\big)(x_{-1}\wedge x_{l_1}\wedge\cdots\wedge x_{l_{s-1}}\wedge x_{-1}\wedge x_{l_{s+1}}\wedge \cdots \wedge x_{l_n})=0,
$
\\we obtain that

$\bar{\Delta}\bar{\mu}(x_{j_1}, \cdots, x_{j_{n+1}})=\sum\limits_{s=1}^{n+1}\sum\limits_{k=1}^{n+1}(-1)^{n+1-k}\rho^{\bar{\mu}}_s(x_{j_1},\cdots, \widehat{x_{j_k}}, \cdots, x_{j_{n+1}})\bar{\Delta}(x_{j_k})=0.$

Therefore, $(\bar L, \bar{\mu}, \bar{\Delta})$ is an $(n+1)$-Lie bialgebra. The proof is complete.

\end{proof}

From Theorem \ref{thm:cobialg} and Theorem \ref{thm:biext}, we have the following result.

\begin{coro}\label{cor:2}

Let $(L, \mu, \Delta)$ be an  $n$-Lie bialgebra, elements $x_0, x_{-1}$ be not contained in $L$, $B: L^*\otimes L^* \rightarrow \mathbb F$ be a bilinear symmetric function  satisfying ad$_{\Delta^*}$-invariant.
Then $(\bar L, \bar{\mu}, \bar{\Delta}) $ is an $(m+2)$-dimensional $(n+1)$-Lie bialgebra, where $\bar L=L\oplus \mathbb F x_0\oplus \mathbb F x_{-1}$,
$(\bar L, \bar{\mu})$ is an $(m+2)$-dimensional $(n+1)$-Lie algebra in Lemma \ref{lem:exten11}, and  $((\bar L)^*, \bar{\Delta^*})$  is an $(m+2)$-dimensional $(n+1)$-Lie algebra
in Lemma \ref{lem:Lexten}.

\end{coro}

In the case $n=2$, let $(L, \mu, \Delta)$ be an $m$-dimensional  Lie bialgebra with a basis $x_1, \cdots, x_m$, $B: L\otimes L \rightarrow \mathbb F$ be an $ad_{\mu}$-invariant symmetric bilinear form with $B(x_i, x_j)=b_{ij}\in \mathbb F$, $1\leq i, j\leq m$.
Suppose
$$
\mu(x_i, x_j)=\sum\limits_{k=1}^m c_{ij}^k x_k, ~~ \Delta(x_k)=\sum\limits_{1\leq i_1<i_2\leq m} a^{i_1i_2}_kx_{i_1}\wedge x_{i_2}, ~~ c_{ij}^k, a^{i_1i_2}_k\in \mathbb F.
$$

Then $3$-Lie algebra $(\bar L, \bar{\mu})$ in Lemma \ref{lem:Lexten} and $(\bar{L^*}, \bar{\Delta^*})$ in  Lemma \ref{lem:exten11}
with  the following multiplication

$$\bar{\mu}(x_{i_1}, x_{i_2}, ~x_{0})=\sum\limits_{l=-1}^m c_{i_1i_20}^lx_l=\mu(x_{i_1}, x_{i_2})=\sum\limits_{t=1}^mc_{i_1i_2}^tx_t, ~1\leq i_1,
\ldots, i_{n}\leq m;
$$

$$
\bar{\mu}(x_{i_1},x_{i_2}, x_{-1})=\sum\limits_{l=-1}^m c_{i_1i_2(-1)}^lx_l=0,
~~ ~ 0\leq  i_1, \ldots, i_n\leq m;
$$

$$
\bar{\mu}(x_{i_1}, x_{i_2} ~ x_{i_3})=\sum\limits_{l=-1}^m c_{i_1i_2i_3}^lx_l=B(\mu(x_{i_1},
x_{i_2}), x_{i_3})x_{-1}=\sum\limits_{t=1}^m c_{i_1i_2}^tb_{ti_3}x_{-1}, ~1\leq i_1, \ldots,
i_{n+1}\leq m.
$$
And for $-1\leq l, j_1, j_2\leq m,$ $ 1\leq i_1, i_2, t\leq m,$
$$
c_{j_1j_20}^{-1}=c_{j_1j_20}^{0}=0,
~~c_{i_1i_20}^{t}=c_{i_1i_2}^{t}, ~~  ~c_{j_1j_2(-1)}^{l}=0,
$$

$$
 ~c_{i_1i_2i_3}^{l}=0, \quad  0\leq l\leq m, \quad  ~c_{i_1i_2i_3}^{-1}=\sum\limits_{t=1}^mc_{i_1i_2}^tb_{ti_3}. ~
$$

$$\bar{\Delta^*}(x^{i_1}, x^{i_2}, ~x^{-1})=\sum\limits_{l=-1}^m a^{i_1i_2(-1)}_lx^l=\Delta^*(x_{i_1}, x_{i_2})=\sum\limits_{t=1}^ma^{i_1i_2}_tx^t, ~1\leq i_1,
\ldots, i_{n}\leq m;
$$

$$
\bar{\Delta^*}(x^{i_1},x^{i_2}, x^{0})=\sum\limits_{l=-1}^m a^{i_1i_20}_lx^l=0,
~~ ~ 0\leq  i_1, \ldots, i_n\leq m;
$$

$$
\bar{\Delta^*}(x^{i_1}, x^{i_2} ~ x^{i_3})=0, ~1\leq i_1, \ldots,
i_{n+1}\leq m.
$$

Therefore,

$$\bar{\Delta}(x_k)=\sum\limits_{-1\leq j_1j_2j_3\leq m}a_k^{j_1j_2j_3}x_{j_1}\otimes x_{j_2}\otimes x_{j_3}
=\sum\limits_{(-1)i_1i_2}a_{k}^{i_1i_2}x_{-1}\otimes x_{i_1}\otimes x_{i_2}, 1\leq i_1, i_2\leq m,
$$

$$
\bar{\Delta}(x_0)=0,
\bar{\Delta}(x_{-1})=\sum\limits_{-1\leq j_1j_2j_3\leq m}a_{-1}^{j_1j_2j_3}x_{j_1}\otimes x_{j_2}\otimes x_{j_3}=0.
$$

We obtain for all  $-1\leq l, j_1, j_2, j_3\leq m,$ $ 1\leq i_1, i_2, i_3, t\leq m,$

$$
a^{j_1j_2j_3}_{-1}=a^{j_1j_2j_3}_{0}=0, \quad a^{j_1j_2 0}_{l}=a^{j_1j_2 (-1)}_{0}=0,
$$
$$
a^{j_1j_2 (-1)}_{t}=a_t^{j_1j_2}, \quad
a^{i_1i_2i_3}_{t}=0, \quad 1\leq t\leq m.
$$

\section{ Structures of $n$-Lie bialgebra $(A_n, \mu, \Delta)$}
\mlabel{sec:simbialg}

Ling in \cite{L} proved that there exists only one simple $n$-Lie algebra over the field $\mathbb F$ of complex numbers, that is $(n+1)$-dimensional $n$-Lie algebra, is denoted by $(A_n, \mu)$, or simply is denoted by $A_n$.  In this section, we study  bialgebra structures on the simple  $n$-Lie algebra $(A_n, \mu)$ over the field $\mathbb F$ of complex numbers. First we give the classification theorem given in
paper \cite{BSZ1}.

\begin{lemma} \cite{BSZ1} \label{lem:classnl} Let $A$ be an
$(n+1)$-dimensional $n$-Lie algebra  and  $e_{1},$ $ e_{2},$
  $ \cdots,$ $ e_{n+1}$ be a basis of $A$ ($n\geq 3$). Then one and only one
  of the following possibilities holds up
to isomorphisms:

\noindent $(a)$ ~If $\dim A^{1}=0$, then $A$ is an
abelian $n$-Lie algebra.

\noindent  $(b)$ ~If $\dim A^{1}=1$ and let $A^{1} =
\mathbb F e_{1}$, then in the case  $A^{1}\subseteq Z(A)$,

$ (b_1). ~ \mu(e_{2}, \cdots, e_{n+1}) = e_{1}.$

In the case that $A^{1}$ is not contained in $Z(A)$,

$ (b_2). ~ \mu(e_{1}, \cdots, e_{n})=e_{1}.$

\noindent  $(c)$ ~If $\dim A^{1}=2$ and let $A^{1}= {\mathbb F} e_{1}+ {\mathbb F} e_{2}$, then

$(c_{1}). ~\mu( e_{2}, \cdots, e_{n+1}) = e_{1},
{[}e_{1}, e_{3}, \cdots, e_{n+1}] = e_{2};$

$ (c_{2}). ~~ \mu(e_{2}, \cdots, e_{n+1}) =\alpha e_{1}+ e_{2}, ~{[}e_{1}, e_{3}, \cdots, e_{n+1}] = e_{2};
$

$
 (c_{3}). ~~
\mu(e_{1}, e_{3}, \cdots, e_{n+1}) = e_{1}, {[}e_{2}, \cdots, e_{n+1}] = e_{2},$
where $~ \alpha \in \mathbb F$ and $ \alpha \neq 0.$

\noindent  $(d)$ ~If  $\dim A^{1}=r$, $ 3\leq r\leq
n+1$, let $A^{1}= {\mathbb F} e_{1}+ {\mathbb F} e_{2}+ \cdots + {\mathbb F} e_{r}$. Then

$
(d_r). ~\mu(e_1, \cdots, \hat{e}_i, \cdots, e_{n+1})=e_i, ~1\le i \le r,
$
 where symbol
$\hat{e}_i$
 means that  $e_{i}$ is omitted.

 In the case $r=n+1$, $A$ is the $(n+1)$-dimensional simple $n$-Lie algebra, which is denoted by  $(A_n, \mu),$ or simply $A_n$.

 \end{lemma}

Let  $\Delta: A\rightarrow  \underbrace{A \wedge \cdots \wedge A}_n $ be a linear mapping, set
$$\Delta (e_i)=\sum\limits_{j=1}^{n+1} a_{ij}~~ e_1 \wedge \cdots \wedge \widehat{e_j} \wedge \cdots \wedge e_{n+1}, ~~ 1\leq i\leq n+1,  \quad a_{ij}\in \mathbb F.$$

Suppose that $A^*$ is the dual space of $A$, and $e^1, \cdots, e^{n+1} $ is the dual basis of $e_1, \cdots, e_{n+1}$, that is, $\langle e^i, e_j\rangle=\delta_{ij}, 1\leq i, j\leq n+1,$ and  $\Delta^*: A^{*\otimes n} \rightarrow A^*$ is the dual mapping of $\Delta$. Then by Eq.\eqref{eq:CCL} we have
$$
\Delta^*(e^1, \cdots, \widehat{e^j}, \cdots, e^{n+1})=\sum\limits_{i=1}^{n+1} a_{ij} e^j, ~~ 1\leq j \leq n+1.
$$

Denote $$\bar{e}^j=(-1)^{n+j+1}\Delta^*(e^1, \cdots, \widehat{e^j}, \cdots, e^{n+1})=(-1)^{n+j+1}\sum\limits_{i=1}^{n+1} a_{ij} e^j=\sum\limits_{i=1}^{n+1} b_{ij} e^j,$$
that is, $b_{ij}=(-1)^{n+j+1}a_{ij}$, $1\leq i, j\leq n+1$. We obtain a matrix $B=(b_{ij})_{(n+1)\times (n+1)}$, and
$$
(\bar e^1, \cdots, \bar e^{n+1})=( e^1, \cdots,  e^{n+1}) B.
$$

\begin{lemma}\label{lem:simple}

$(A, \Delta)$ is an $n$-Lie coalgebra with $R(\Delta)\geq 3$ if and only if $B=(b_{ij})$ is a symmetric matrix with rank $R(B)\geq 3.$

\end{lemma}

\begin{proof}

The result follows from  Theorem 3 in paper \cite{F}, and Theorem \ref{thm:LandC} directly.

\end{proof}

\begin{theorem}\label{thm:class}
The triple $(A, \mu, \Delta)$ (or the pair $(A_n, \Delta)$ ) is an $n$-Lie bialgebra if and only if $R(\Delta)=0$ ($(A^*, \Delta^*)$  is abelian), or $R(\Delta)=2$ and  $(A^*, \Delta^*)$ is the $n$-Lie algebra of type $(c_3)$
in Lemma \ref{lem:classnl}
\end{theorem}

\begin{proof}    Since $A_n$ is the simple $n$-Lie algebra, by Lemma \ref{lem:classnl},

\begin{align*}
\Delta\mu(x_2,\cdots,x_{n+1})=\Delta(x_1)=&\sum\limits_{k=1}^{n+1}a_{1k}x_1\wedge\cdots\wedge\widehat{x_k}\wedge\cdots\wedge x_{n+1}\\
=&a_{11}x_2\wedge\cdots\wedge x_{n+1}+\sum\limits_{k=2}^{n+1}a_{1k}x_1\wedge\cdots\wedge\widehat{x_k}\wedge\cdots\wedge x_{n+1}\\
=&a_{11}x_2\wedge\cdots\wedge x_{n+1}+\sum\limits_{k=1}^{n}a_{1k+1}x_1\wedge\cdots\wedge\widehat{x_{k+1}}\wedge\cdots\wedge x_{n+1},
\end{align*}
and
\begin{align*}
&\sum\limits_{s=1}^{n}\sum\limits_{k=1}^{n}(-1)^{n-k}\rho_{s}^{\mu}(x_2,\cdots,\widehat{x_{k+1}},\cdots,x_{n+1})\Delta(x_{k+1})\\
&=\sum\limits_{s=1}^{n}\sum\limits_{k=1}^{n}(-1)^{n-k}\rho_{s}^{\mu}(x_2,\cdots,\widehat{x_{k+1}},\cdots,x_{n+1})(\sum\limits_{r=1}^{n+1}a_{k+1r} x_1\wedge\cdots\wedge\widehat{x_r}\wedge\cdots\wedge x_{n+1})\\
&=\sum\limits_{s=1}^{n}\sum\limits_{k=1}^{n}(-1)^{n-k}\rho_{s}^{\mu}(x_2,\cdots,\widehat{x_{k+1}},\cdots,x_{n+1})(a_{k+11}x_2\wedge\cdots\wedge x_{n+1}
\\
&+\sum\limits_{r=2}^{n+1}a_{k+1r} x_1\wedge\cdots\wedge\widehat{x_r}\wedge\cdots\wedge x_{n+1})\\
&=\sum\limits_{k=1}^{n}(-1)^{k-1}a_{k+11}x_1\wedge\cdots\wedge\widehat{x_{k+1}}\wedge\cdots\wedge x_{n+1}\\
&+\sum\limits_{k=1}^{n}(-1)^{n-k}\sum\limits_{r=2}^{n+1}a_{k+1r}ad_{\mu}(x_2,\cdots,\widehat{x_{k+1}},\cdots,x_{n+1})(x_1)
\wedge\cdots\wedge\widehat{x_r}\wedge\cdots\wedge x_{n+1}\\
&+\sum\limits_{s=2}^{n}\sum\limits_{k=1}^{n}(-1)^{n-k}\rho_{s}^{\mu}(x_2,\cdots,\widehat{x_{k+1}},\cdots,x_{n+1})(\sum\limits_{r=2}^{n+1}a_{k+1r} x_1\wedge\cdots\wedge\widehat{x_r}\wedge\cdots\wedge x_{n+1})\\
&=\sum\limits_{k=1}^{n}(-1)^{k-1}a_{k+11}x_1\wedge\cdots\wedge\widehat{x_{k+1}}\wedge\cdots\wedge x_{n+1}\\
&+\sum\limits_{k=1}^{n}a_{k+1k+1}x_2\wedge\cdots\wedge x_{n+1}.
\end{align*}
Then  $\Delta\mu(x_2,\cdots,x_{n+1})=\sum\limits_{s=1}^{n}\sum\limits_{k=1}^{n}(-1)^{n-k}\rho_{s}^{\mu}(x_2,\cdots,\widehat{x_{k+1}},\cdots,x_{n+1})\Delta(x_{k+1}),$
if and only if

$$\left\{
\begin{array}{l}
a_{11}=\sum\limits_{k=2}^{n}a_{kk},\\
a_{1k}=(-1)^{k+1-1}a_{k1}, \quad for \quad 2\leq k\leq n+1.
\end{array}
\right.
$$

For all $2\leq i\leq n+1$, from

$\Delta\mu(x_1,\cdots,\widehat{x_i},\cdots, x_{n+1})=\Delta(x_i)=\sum\limits_{k=1}^{n+1}a_{ik}x_1\wedge\cdots\wedge\widehat{x_k}\wedge\cdots\wedge x_{n+1},$

and

\vspace{2mm}$
\sum\limits_{s=1}^{n}\sum\limits_{k=1}^{i-1}(-1)^{n-k}\rho_{s}^{\mu}(x_1,\cdots,\widehat{x_{k}},\cdots,\widehat{x_{i}},\cdots,x_{n+1})\Delta(x_k)$
\\
$+\sum\limits_{s=1}^{n}\sum\limits_{k=i}^{n}(-1)^{n-k}\rho_{s}^{\mu}(x_1,\cdots,\widehat{x_{i}},\cdots,\widehat{x_{k+1}},\cdots,x_{n+1})\Delta(x_{k+1})$
\\
$=\sum\limits_{s=1}^{n}\sum\limits_{k=1}^{i-1}(-1)^{n-k}\rho_{s}^{\mu}(x_1,\cdots,\widehat{x_{k}},\cdots,\widehat{x_{i}},\cdots,x_{n+1})
(\sum\limits_{r=1}^{n+1}a_{kr}x_1\wedge\cdots\wedge\widehat{x_r}\wedge\cdots\wedge x_{n+1})$
\\
$+\sum\limits_{s=1}^{n}\sum\limits_{k=i}^{n}(-1)^{n-k}\rho_{s}^{\mu}(x_1,\cdots,\widehat{x_{i}},\cdots,\widehat{x_{k+1}},\cdots,x_{n+1})
(\sum\limits_{r=1}^{n+1}a_{k+1r}x_1\wedge\cdots\wedge\widehat{x_r}\wedge\cdots\wedge x_{n+1})$
\\
$=\sum\limits_{s=1}^{n}\sum\limits_{k=1}^{i-1}(-1)^{n-k}\rho_{s}^{\mu}(x_1,\cdots,\widehat{x_{k}},\cdots,\widehat{x_{i}},\cdots,x_{n+1})
(\sum\limits_{r=1}^{i-1}a_{kr}x_1\wedge\cdots\wedge\widehat{x_r}\wedge\cdots\wedge x_{n+1})$
\\
$+\sum\limits_{s=1}^{n}\sum\limits_{k=1}^{i-1}(-1)^{n-k}\rho_{s}^{\mu}(x_1,\cdots,\widehat{x_{k}},\cdots,\widehat{x_{i}},\cdots,x_{n+1})
(\sum\limits_{r=i}^{n+1}a_{kr}x_1\wedge\cdots\wedge\widehat{x_r}\wedge\cdots\wedge x_{n+1})$
\\
$+\sum\limits_{s=1}^{n}\sum\limits_{k=i}^{n}(-1)^{n-k}\rho_{s}^{\mu}(x_1,\cdots,\widehat{x_{i}},\cdots,\widehat{x_{k+1}},\cdots,x_{n+1})
(\sum\limits_{r=1}^{i-1}a_{k+1r}x_1\wedge\cdots\wedge\widehat{x_r}\wedge\cdots\wedge x_{n+1})$
\\
$+\sum\limits_{s=1}^{n}\sum\limits_{k=i}^{n}(-1)^{n-k}\rho_{s}^{\mu}(x_1,\cdots,\widehat{x_{i}},\cdots,\widehat{x_{k+1}},\cdots,x_{n+1})
(\sum\limits_{r=i}^{n+1}a_{k+1r}x_1\wedge\cdots\wedge\widehat{x_r}\wedge\cdots\wedge x_{n+1})$
\\
$
=\sum\limits_{k=1}^{i-1}a_{kk}x_1\wedge\cdots\wedge\widehat{x_i}\wedge\cdots x_{n+1}
+\sum\limits_{k=1}^{i-1}a_{ki}(-1)^{i-k-1}x_1\wedge\cdots\widehat{x_k}\wedge\cdots\wedge x_{n+1}+0$
 \\
$+\sum\limits_{k=i}^{n}a_{k+1k+1}x_1\wedge\cdots\wedge\widehat{x_i}\wedge\cdots\wedge x_{n+1}
+\sum\limits_{k=i}^{n}(-1)^{k+i}a_{k+1i}x_1\wedge\cdots\wedge\widehat{x_{k+1}}\wedge x_{n+1}$
\\
$
=\sum\limits_{k=1}^{i-1}a_{kk}x_1\wedge\cdots\wedge\widehat{x_i}\wedge\cdots x_{n+1}+\sum\limits_{k=i}^{n}a_{k+1k+1}x_1\wedge\cdots\wedge\widehat{x_i}\wedge\cdots\wedge x_{n+1}
$
\\
$+\sum\limits_{k=1}^{i-1}a_{ki}(-1)^{i-k-1}x_1\wedge\cdots\widehat{x_k}\wedge\cdots\wedge x_{n+1}
+\sum\limits_{k=i+1}^{n+1}a_{ki}(-1)^{k+i+1}x_1\wedge\cdots\widehat{x_k}\wedge\cdots\wedge x_{n+1}
$
\\
$
=\sum\limits_{k=1}^{i-1}a_{kk}x_1\wedge\cdots\wedge\widehat{x_i}\wedge\cdots x_{n+1}+\sum\limits_{k=i}^{n}a_{k+1k+1}x_1\wedge\cdots\wedge\widehat{x_i}\wedge\cdots\wedge x_{n+1}
$
\\
$+\sum\limits_{k=1}^{n}a_{ki}(-1)^{k-i-1}x_1\wedge\cdots\widehat{x_k}\wedge\cdots\wedge x_{n+1}.
$

\vspace{2mm} We obtain that

\vspace{2mm}$\Delta\mu(x_1,\cdots,\widehat{x_i},\cdots, x_{n+1})$
\\
$=\sum\limits_{s=1}^{n}\sum\limits_{k=1}^{i-1}(-1)^{n-k}\rho_{s}^{\mu}(x_1,\cdots,\widehat{x_{k}},\cdots,\widehat{x_{i}},\cdots,x_{n+1})\Delta(x_k)$
\\
$+\sum\limits_{s=1}^{n}\sum\limits_{k=i}^{n}(-1)^{n-k}\rho_{s}^{\mu}(x_1,\cdots,\widehat{x_{i}},\cdots,\widehat{x_{k+1}},\cdots,x_{n+1})\Delta(x_{k+1})$,~~
if and only if

$a_{ii}=\sum\limits_{k=1,k\neq i}^{n+1}a_{kk},$ \quad
$a_{ik}=(-1)^{k-i-1}a_{ki},$ \quad  for all $2\leq i\leq n+1, \quad 2\leq k\leq n+1.
$

Summarizing above discussion, we obtain that $\mu, \Delta $ satisfy Eq.\eqref{eq:bialg1} if and only if

\begin{equation}\label{eq:c}
a_{kk}=0, \quad  a_{ij}=(-1)^{i+j+1}a_{ji}, \quad \forall ~ 1\leq k\leq n+1, \quad 1\leq i\neq j\leq n+1.
\end{equation}

Therefore, for all $~ 1\leq k\leq n+1, \quad 1\leq i\neq j\leq n+1, $

\begin{equation}\label{eq:bij}
b_{kk}=(-1)^{n+k+1}a_{kk}=0, \quad  b_{ij}=(-1)^{n+j+1}a_{ij}=(-1)^{n+j+1}(-1)^{i+j+1}a_{ji}=-b_{ji},
\end{equation}
\\ that is,  $B$ is a skew-symmetrix  matrix. Then  $R(B)\neq 1$, and from Lemma \ref{lem:simple}, $R(B)\leq 2$.
This shows that  there do not exist $n$-Lie bialgebra structures on the simple $n$-Lie algebra $A_n$ with  $R(\Delta)\geq 3$ and $R(\Delta)=1.$

By a direct computation,  the only $n$-Lie bialgebra structures on the simple $n$-Lie algebra $A_n$ are $R(\Delta)=0$,  or $R(\Delta)=2$, and in this case,   the $n$-Lie algebra $(A^*, \Delta^*)$ is the type $(c_3)$.

\end{proof}

\noindent
{\bf Acknowledgements. } The authors would like to thank Professor Chengming Bai for many
valuable suggestions.
The first author was supported in part by the Natural
Science Foundation (11371245) and the Natural
Science Foundation of Hebei Province (A2014201006).

\bibliography{}

\end{document}